\documentclass[11pt]{amsart}

\usepackage{epigamath}

\usepackage[english]{babel}

\numberwithin{equation}{section}

\usepackage{graphics,graphicx}
\usepackage{tikz-cd}
\usetikzlibrary{calc}

\usepackage{xcolor}

\usepackage{enumitem}

\usepackage{booktabs}

\usepackage[utf8]{inputenc}
\usepackage[T1]{fontenc}

\usepackage{url}

\usepackage{amssymb,amsmath,amsthm}

\newtheorem{thm}{Theorem}[section]
\newtheorem{cor}[thm]{Corollary}
\newtheorem{lem}[thm]{Lemma}
\newtheorem{prop}[thm]{Proposition}

\theoremstyle{definition}

\theoremstyle{remark}
\newtheorem{rem}[thm]{Remark}

\DeclareMathOperator{\SL}{SL}

\DeclareMathOperator{\Aut}{Aut}

\newcommand{\bbC}{\mathbb{C}}
\newcommand{\bbQ}{\mathbb{Q}}
\newcommand{\bbZ}{\mathbb{Z}}
\newcommand{\bbP}{\mathbb{P}}
\newcommand{\bbR}{\mathbb{R}}
\newcommand{\bbN}{\mathbb{N}}
\newcommand{\group}{G}
\newcommand{\borel}{B}
\newcommand{\torus}{T}

\newcommand{\rootSystem}{R}
\newcommand{\halfSum}{\varpi}
\newcommand{\isotropy}{H}
\newcommand{\homo}{\group/\isotropy}
\newcommand{\sphericalLattice}{M}
\newcommand{\realSphericalLattice}{M \otimes \bbR}

\newcommand{\dualLattice}{N}
\newcommand{\realDualLattice}{N \otimes \bbR}
\newcommand{\rationalDualLattice}{N \otimes \bbQ}
\newcommand{\valCone}{\mathcal{V}} 
\newcommand{\linValCone}{\operatorname{Lin}(\valCone)}
\newcommand{\complementLinValCone}{\mathcal{W}}
\newcommand{\colorMap}{\varrho}
\newcommand{\variety}{X}
\newcommand{\lineBundle}{L}
\newcommand{\polVar}{(\variety,\lineBundle)}
\newcommand{\globalSection}{s}
\newcommand{\coloredFan}{\mathcal{F}_{\variety}}
\newcommand{\polytope}{\Delta} 
\newcommand{\relevantRoots}{\rootSystem_{\variety}^+}
\newcommand{\weightSection}{\chi}
\newcommand{\divBinv}{\mathcal{P}_{\variety}}
\newcommand{\homoTC}{\homo\times\bbC^*}
\newcommand{\colorMapTC}{\hat{\colorMap}}
\newcommand{\varietyTC}{\mathcal{X}} 
\newcommand{\lineBundleTC}{\mathcal{L}}
\newcommand{\polVarTC}{(\varietyTC,\lineBundleTC)}
\newcommand{\fiberTC}{\variety}
\newcommand{\globalSectionTC}{\hat{s}}
\newcommand{\polytopeTC}{\hat{\polytope}} 
\newcommand{\sphericalLatticeTC}{\hat{\sphericalLattice}} 
\newcommand{\divBinvTC}{\mathcal{P}_{\varietyTC}}
\newcommand{\coloredFanTC}{\mathcal{F}_{\varietyTC}}
\newcommand{\valConeTC}{\valCone\times\bbR}
\DeclareMathOperator{\MNA}{M^{NA}}
\DeclareMathOperator{\JNA}{J^{NA}}
\DeclareMathOperator{\DF}{DF}
\DeclareMathOperator{\LFunctional}{\mathcal{L}}
\DeclareMathOperator{\JFunctional}{\mathcal{J}}
\newcommand{\lebesgue}{\mathop{d\mu}}
\newcommand{\lebesgueBoundary}{\mathop{d\sigma}}

\newcommand{\scalar}{a}
\newcommand{\DHpol}{P}
\newcommand{\coDHpol}{Q}
\DeclareMathOperator{\Int}{Int}
\newcommand{\smoothNormalized}{\mathcal{N}_0^{\infty}}

\newcommand{\lra}{\longrightarrow}

\newcommand{\supth}[1]{\ensuremath{#1^{\mathrm{th}}}}

\EpigaVolumeYear{7}{2023} \EpigaArticleNr{9} \ReceivedOn{August 24, 2022}
\InFinalFormOn{}
\AcceptedOn{December 1st, 2022}

\title{Uniform K-stability of polarized spherical varieties}
\titlemark{Uniform K-stability of polarized spherical varieties}

\author{Thibaut Delcroix}
\address{Univ.\ Montpellier, CNRS, Montpellier, France}
\email{thibaut.delcroix@umontpellier.fr}
\authormark{T.~Delcroix}

\AbstractInEnglish{We express notions of K-stability of polarized spherical varieties in terms of combinatorial data, vastly generalizing the case of toric varieties.  We then provide a combinatorial sufficient condition for \(G\)-uniform K-stability by studying the corresponding convex geometric problem.
Thanks to recent work of Chi Li and a remark by Yuji Odaka, this provides an explicitly checkable sufficient condition for the existence of constant scalar curvature Kähler metrics.  As a side effect, we show that, on several families of spherical varieties, \(G\)-uniform K-stability is equivalent to K-polystability with respect to \(G\)-equivariant special test configurations for polarizations close to the anticanonical bundle.}



\acknowledgement{This research received partial funding from ANR Project FIBALGA ANR-18-CE40-0003 and ANR Project MARGE ANR-21-CE40-0011.}

\contribution{With an appendix by Yuji Odaka}

\begin{document}



\maketitle

\begin{prelims}

\DisplayAbstractInEnglish

\bigskip

\DisplayKeyWords

\medskip

\DisplayMSCclass







\end{prelims}


\newpage

\setcounter{tocdepth}{1}

\tableofcontents


\section{Introduction}

In the seminal article \cite{Don02}, Donaldson initiated a study of the existence of constant scalar curvature K\"ahler metrics on polarized toric manifolds.  There, he notably introduced the general condition of K-stability, thus formulating a precise version of the Yau--Tian--Donaldson conjecture whose aim is to give an algebro-geometric characterization of the existence of constant scalar curvature K\"ahler metrics on general polarized varieties.  Focusing on toric varieties, he further translated the condition of K-stability into a convex geometric problem, and with additional work concluded in \cite{Don09}, he was able to prove the Yau--Tian--Donaldson conjecture for non-singular toric surfaces.  Together with the case of K\"ahler--Einstein metrics, it is still today some of the most convincing evidence for the conjecture.

The conjecture is still open in general, and examples (see \cite{ACGT08}) indicate that the condition of K-stability should be modified slightly to a condition of \emph{uniform} K-stability which has been introduced and refined by several authors \cite{Sz15,Der16,BHJ17,His1}.  However, the work of Donaldson combined with some more recent advances (notably \cite{CC2}) shows that the uniform Yau--Tian--Donaldson conjecture is true for toric manifolds.  A different proof of this fact has been provided very recently by Li \cite{Li}.  Odaka noticed that the recent work of Li essentially proves the uniform Yau--Tian--Donaldson conjecture for polarized spherical manifolds as well (this is explained in more details in his appendix to the present article).

Motivated by Odaka's remark, we translate the uniform K-stability condition into a convex geometric problem for polarized spherical varieties, in terms of the combinatorial data encoding these.  This translation provides a much wider playground than the toric case where one can try to show the (non-uniform) Yau--Tian--Donaldson conjecture for different classes of varieties, or derive explicit combinatorial conditions for the existence of constant scalar curvature K\"ahler metrics.  We will concentrate here on the second goal and provide a combinatorial sufficient criterion for uniform K-stability which applies to a wide range of polarized spherical varieties.  We intend to present progress on the first goal in another article.\footnote{In the preprint \cite{RK1}, the author has applied the present article to prove the Yau--Tian--Donaldson conjecture for cohomogeneity one manifolds and give a simple combinatorial criterion.}  It should be noted, for the reader more familiar with
K-stability for Fano manifolds, that special test configurations are not expected to be enough, and thus no valuative criterion is expected to hold in the general polarized case. Furthermore, to the author's knowledge, before the present paper, the only manifolds proved to admit a cscK (non-K\"ahler--Einstein) metric by direct K-stability arguments were toric surfaces.

To provide the reader with a better flavor of the convex geometric problem associated to K-stability of spherical varieties, let us first recall the case of toric varieties.  A polarized toric variety \(\polVar\) is an ordered pair formed by a complex normal \(n\)-dimensional projective variety \(\variety\) equipped with an effective action of \((\bbC^*)^n\) and a \((\bbC^*)^n\)-linearized ample line bundle \(L\) on \(X\).  Such  data is fully encoded by a convex polytope \(\polytope\) with integral vertices in \(\bbR^n\), and the correspondence is very explicit: the integral points of \(\polytope\) coincide with the different \((\bbC^*)^n\)-weights of the action of \((\bbC^*)^n\) on the space of global holomorphic sections \(H^0\polVar\).  Let \(\lebesgue\) denote the standard Lebesgue measure on \(\bbR^n\), and let \(\lebesgueBoundary\) denote the measure on \(\partial\polytope\) which coincides on each facet of \(\polytope\) with the Lebesgue measure on its affine span \(V\), normalized to give unit
mass to a fundamental region of the lattice \(\bbZ^n\cap V\).  Donaldson shows that the polarized toric variety \(\polVar\) is K-stable if and only if the following functional is positive for any rational piecewise linear concave function \(g\) on \(\polytope\): 
\begin{equation*} 
\LFunctional(g) = 2 \scalar \int_{\polytope} g \lebesgue - \int_{\partial\polytope} g \lebesgueBoundary,
\end{equation*}
where \(\scalar\) is the real number such that \(\LFunctional\) vanishes on constants.

A polarized spherical variety \(\polVar\) is an ordered pair formed by a normal projective variety \(\variety\) equipped with an action of a connected complex reductive group \(\group\) and a \(\group\)-linearized ample line bundle \(\lineBundle\) on \(\variety\), such that a Borel subgroup of \(\group\) acts with an open orbit on \(\variety\).  Such a variety is encoded as well by combinatorial data, including the data of a convex polytope in a real vector space \(\bbR^r\) (of dimension smaller than \(\variety\) in general), but the definition of these is more involved.  Let us just mention the nature of the problem here; we refer to the body of the paper (mainly Section~\ref{sec_spherical}) for precise definitions.  The functional \(\LFunctional\) from the toric case is modified as
\begin{equation}
\label{eqn_LFunctional}
\LFunctional(g) = \int_{\polytope} 2 g (\scalar \DHpol - \coDHpol) \lebesgue - \int_{\partial\polytope} g \DHpol \lebesgueBoundary,  
\end{equation}
where \(\polytope\) is some convex polytope with rational vertices in \(\bbR^r\), the measure \(\lebesgue\) and \(\lebesgueBoundary\) are defined as in the toric case, \(\DHpol\) and \(\coDHpol\) are polynomials, and the scalar \(\scalar\) is still such that \(\LFunctional\) vanishes on constants.  Let us consider as well the functional
\begin{equation*}
\JFunctional(g) = \int_{\polytope} \left(\max_{\polytope} g - g\right) \DHpol \lebesgue.
\end{equation*}
We add that the polynomial \(\DHpol\) is positive on the interior of \(\polytope\), so that the functional \(\JFunctional\) plays a role of (semi)norm: it is non-negative and vanishes only on constant concave functions.  Our first main statement translates conditions of K-stability into conditions on the functionals above (see Section~\ref{sec_stability}  for a recall on these notions).  Before stating these, let us note that there is additional combinatorial data associated to a polarized spherical variety: its valuation cone, which may for now simply be interpreted as the data of some full-dimensional convex cone \(\valCone\) in \((\bbR^n)^*\) centered at \(0\). We denote by \(\linValCone\) the linear part of \(\valCone\), that is, the largest linear subspace contained in \(\valCone\).

\begin{thm}
\label{thm_Kstab}
\samepage{
Let \(\polVar\) be a polarized \(\group\)-spherical variety. 
Then 
\begin{enumerate}
\item it is \(\group\)-equivariantly K-polystable if and only if \(\LFunctional(g) \geq 0\) for any rational piecewise linear concave function \(g\colon \polytope \to \bbR\) with slopes in \(\valCone\), and equality holds only if \(g\in \linValCone\);
\item and it is \(\group\)-uniformly K-stable if and only if there exists an \(\varepsilon > 0\) such that
\[ \LFunctional(g) \geq \varepsilon \inf_{l \in \linValCone} \JFunctional(g + l) \]
for all rational piecewise linear concave functions \(g\colon \polytope \to \bbR\) with slopes in \(\valCone\).
\end{enumerate}
}
\end{thm}

We will further give another formulation of the condition of uniform K-stability and use this new formulation to obtain a sufficient criterion.  We now assume that \(\polytope\) contains the origin in its interior.  Let \(\complementLinValCone\) be a linear complement to \(\linValCone\).  Let \(\smoothNormalized\) be the set of all continuous concave functions \(f\) on \(\polytope\), smooth on the interior with differentials in \(\valCone\), such that \(\max f = 0\) and \(d_0f \in \complementLinValCone\).  There exist a non-positive integrable function \(J\) on \(\polytope\) with negative integral and an integrable function \(K\) with integral zero on \(\polytope\) such that the functional \(\LFunctional_s\) defined on \(\smoothNormalized\) by
\begin{equation}
\label{eqn_LFunctional_s}
\LFunctional_s(f) = \int_{\polytope} (f(x) K(x) + d_xf(x) J(x)) \lebesgue(x) 
\end{equation}
coincides with \(\LFunctional\) on \(\smoothNormalized\).  Here, \(d_xf\) denotes the differential of a function \(f\) at \(x\), so that \(d_xf(x)\) is equal to \(\nabla f(x) \cdot x\) if one chooses a Euclidean norm on \(\bbR^n\).  We refer to Section~\ref{sec_new} for the details on the definition of \(K\) and \(J\).

\begin{thm}
\label{thm_sufficient}
Assume that \(K + J \leq 0\). 
Let \(b\) be the element of \(\polytope\) defined by 
\begin{equation*}
\int_{\polytope} (x-b) (K(x)+J(x)) \lebesgue(x) = 0.
\end{equation*}
If\, \(b\) is in the relative interior of\, \(-\valCone^{\vee}\), then \((X,L)\) is \(G\)-uniformly K-stable.
\end{thm}

Despite its simplicity, the above theorem actually provides an explicit and tractable condition to check for a huge family of polarized spherical varieties.  The main evidence for this follows from two facts.  The first one is that, when specialized to the situation coming from an anticanonically polarized spherical variety, all assumptions but the condition on \(b\) are automatically satisfied, and that condition translates to the criterion for K-polystability with respect to \(\group\)-equivariant special test configurations obtained in \cite{DelKSSV}.  In fact, as we will show in Section~\ref{sec_full}, the condition on \(b\) translates to \(\group\)-equivariant K-polystability with respect to special test configurations in all cases.

The second evidence is that we can prove in many cases that if we fix the variety \(\variety\) but vary the polarization near the anticanonical polarization, the condition on \(K+J\) is open. The condition \(K+J\leq 0\) obviously does not look open, and an easier way to show openness is to rely on a stronger condition \(K+J\leq c < 0\), but it is not always possible for spherical varieties; instead we must understand well how \(K+J\) vanishes or converges to zero.  These two facts together show that one obtains an explicit condition to check for the 
existence of constant scalar curvature K\"ahler metrics on smooth Fano varieties, for polarizations close to the anticanonical one (again, for an explicit range).  We will prove this for toroidal horospherical varieties, as well as for (non-Hermitian) symmetric varieties, but we expect this to hold much more generally, and our proof easily adapts to different situations.

Finally, we note that the condition on \(b\) consists of \(\dim(\linValCone)\) closed conditions and \(\dim(\valCone)-\dim(\linValCone)\) open conditions.  In particular, on a K\"ahler--Einstein Fano symmetric (non-Hermitian) manifold such that \(\linValCone=\{0\}\), our theorem shows that there exist cscK metrics in an explicit neighborhood of the anticanonical line bundle.

The paper is organized as follows.  Section~\ref{sec_stability} is devoted to the recollection of notions on K-stability.  Section~\ref{sec_spherical} summarizes key combinatorial properties of spherical varieties.  In Section~\ref{sec_test_config}, we associate to a spherical test configuration a piecewise linear concave function and translate the effect of twisting a test configuration in terms of this function.  We then express, in Section~\ref{sec_NA}, the non-Archimedean functionals for spherical test configurations as functionals on the associated concave functions.  We will show how Theorem~\ref{thm_sufficient} applies to K-stability in Section~\ref{sec_new}; we then prove this theorem in Section~\ref{sec_proof_sufficient}.  We provide a full statement of the sufficient condition for \(\group\)-uniform K-stability thus obtained in Section~\ref{sec_full}.  The remaining three sections are devoted to illustrating the applicability of the sufficient condition.  In
Section~\ref{sec_example}, we apply our criterion to the blowup of the three-dimensional quadric at a one-dimensional subquadric and obtain from this the existence of cscK metrics in an explicit neighborhood of the anticanonical line bundle.  We show that the criterion always applies for the anticanonical line bundle on Fano manifolds in Section~\ref{sec_Fano}.  In the final Section~\ref{sec_open}, we provide examples of statements to the effect that in a wide range of situations, our criterion shows that \(\group\)-uniform K-stability is equivalent to \(\group\)-equivariant K-polystability with respect to special test configurations for polarizations close to the anticanonical line bundle.  The appendix by Yuji Odaka shows that for non-singular spherical varieties, \(\group\)-uniform K-stability is equivalent to existence of cscK metrics.

\subsection*{Acknowledgments}
Y.\,O. would like to thank Y.~Gongyo, C.~Li and S.~Okawa for the nice interactions.

\section{Background on K-stability}
\label{sec_stability}

Our references for this section are \cite{BHJ17,His1}. We recall the main notions for the reader's convenience.

Let \(\group\) be a complex reductive group.  Let \(\polVar\) be a \(\group\)-polarized variety.  A (normal, ample) \emph{\(\group\)-equivariant test configuration} for \(\polVar\) consists of the data of a normal \((\group \times \bbC^*)\)-variety \(\varietyTC\), a \((\group \times \bbC^*)\)-linearized ample line bundle \(\lineBundleTC\) on \(\varietyTC\), and a \(\bbC^*\)-equivariant flat morphism \(\pi\colon \polVarTC \to \bbC\) whose fiber \((\fiberTC_1,\lineBundleTC_1)\) over \(1\) is \(\group\)-equivariantly isomorphic to \((\variety, \lineBundle^r)\) for some \(r \in \bbZ_{>0}\).  If the (scheme-theoretic) central fiber is normal, then the test configuration is called \emph{special}.  If the total space of the test configuration is (\(\group\)-equivariantly) isomorphic to \(\variety \times \bbC\), then the test configuration is called a \emph{product test configuration}.

The numerical invariants associated to a test configuration \(\polVarTC\) may be defined in terms of the central fiber \((\fiberTC_0,\lineBundleTC_0)\) as follows.  First note that it is equipped with a \(\bbC^*\)-action induced by the action on the test configuration.  For \(k \in \bbN\), let \(d_k\) denote the dimension \(\dim H^0(\fiberTC_0,\lineBundleTC_0^k)\), let \(\lambda_{1,k},\ldots,\lambda_{d_k,k}\) denote the weights of the \(\bbC^*\)-action on \(H^0(\fiberTC_0,\lineBundleTC_0^k)\), and let \(w_k\) denote the sum of the \(\lambda_{i,k}\).  The quotient \(\frac{w_k}{k d_k}\) admits an expansion in powers of \(k\) at infinity; we will be interested in the first two terms:
\[ \frac{w_k}{k d_k} = F_0 + F_1 k^{-1} + o\left(k^{-1}\right). \]

The \emph{non-Archimedean \(J\)-functional} of a test configuration \(\polVarTC\) is 
\[ \JNA\polVarTC = \sup \{ \lambda_{i,k}/k \mid k \in \bbZ_{>0}, 1 \leq i \leq d_k \} - F_0, \]
and the \emph{Donaldson--Futaki invariant} is 
\[ \DF\polVarTC = -F_1. \]
It is often more convenient to work with the \emph{non-Archimedean Mabuchi functional} \(\MNA\) instead of the Donaldson--Futaki invariant.  Indeed, the latter does not vary linearly with base changes of the form \(z \mapsto z^m\) on the test configurations, while the former does.  Their values coincide when the central fiber is reduced, and the Donaldson--Futaki invariant is always greater than or equal to the non-Archimedean Mabuchi functional.  Furthermore, given any test configuration, there exists a base change such that the resulting test configuration has reduced central fiber, and base change preserves \(\group\)-equivariance.

The \(\group\)-polarized variety \(\polVar\) is called \emph{\(\group\)-equivariantly K-semistable} if \(\MNA\polVarTC \geq 0\) for all \(\group\)-equivariant test configurations, and \emph{\(\group\)-equivariantly K-polystable} if furthermore \(\MNA\polVarTC = 0\) if and only if the test configuration is a product test configuration.  We also use the self-explaining
terminology \emph{\(\group\)-equivariantly K-polystable with respect to special test configurations} and abbreviate it to \emph{\(\group\)-stc K-polystable}.

The total space of a given \(\group\)-equivariant test configuration \(\polVarTC\) may actually be the underlying total space of several different \(\group\)-equivariant test configurations.  Let \(F\) be the group of \((\group\times\bbC^*)\)-equivariant automorphisms of \(\varietyTC\).  It contains the factor \(\bbC^*\).  Let \(\mathfrak{Y}(F)\) denote the set of one-parameter subgroups of \(F\).  Let \(\beta\) denote the character of \(\bbC^*\) of weight \(1\), which is identified with a character of \(F\).  Any one-parameter subgroup \(\eta\in \mathfrak{Y}(F)\) such that \(\beta(\eta)=1\) defines a \(\bbC^*\)-action on \(\polVarTC\), a projection to \(\bbC\), and hence a \(\group\)-equivariant test configuration \(\polVarTC_{\eta}\), in general different from the initial \(\polVarTC\).  The test configuration \(\polVarTC_{\eta}\) is called the \emph{twist by \(\eta\) of the test configuration \(\polVarTC\)}.  Since we can work up to base change, the same notion makes sense for any
element of \(\eta\in\mathfrak{Y}(F)\otimes\bbQ\) with \(\beta(\eta)>0\). One can actually extend the definition to irrational \(\eta\in \mathfrak{Y}(F)\otimes\bbR\) though it is not needed for our paper.

Without involving a group \(\group\), 
\emph{$($J-$)$uniform K-stability} is defined as the existence of a positive constant \(\varepsilon > 0\) such that for all test configurations,
\[ \MNA\polVarTC \geq \varepsilon \JNA\polVarTC. \]
The polarized variety \(\polVar\) is called \emph{\(\group\)-uniformly K-stable} if there exists a positive constant \(\varepsilon >0\) such that for all \(\group\)-equivariant test configurations, 
\[ \MNA\polVarTC \geq \varepsilon \inf_{\eta \in \mathfrak{Y}(F)\otimes\bbQ, ~ \beta(\eta) = 1} \JNA\polVarTC_{\eta}. \]
In other words and up to base changes, on the right-hand side, instead of the \(\JNA\) of the test configuration, we consider the infimum of \(\JNA\) over all test configurations with \((\group\times\bbC^*)\)-isomorphic polarized total space but different projections to \(\bbC\).

\section{Background on polarized spherical varieties}
\label{sec_spherical}

\subsection{Main notions}\label{sec31}
Our references for this section are \cite{Kno91,Bri89}. We recall the main notions for the reader's convenience.

Let \(\group\) be a complex connected reductive group.  We fix a choice of a Borel subgroup \(\borel \subset \group\) and a choice of a maximal torus \(\torus \subset \borel\).  Let \(\homo\) be a \emph{spherical \(\group\)-homogeneous space}, that is, such that \(\borel\) acts on \(\homo\) with an open orbit.  The \emph{spherical lattice} \(\sphericalLattice\) of \(\homo\) is the subgroup of the group of characters of \(\borel\) consisting of the weights of \(\borel\)-eigenfunctions in the field \(\bbC(\homo)\).  We denote by \(\dualLattice\) the dual lattice: \(\dualLattice = \operatorname{Hom}(\sphericalLattice,\bbZ)\).  Since \(\borel\) has an open orbit, the value of a \(\borel\)-invariant valuation on a \(\borel\)-eigenfunction depends only on the eigenvalue, which is an element of \(\sphericalLattice\).  We denote by \(\colorMap\) the map from the set of \(\borel\)-invariant valuations of \(\bbC(\homo)\) to \(\realDualLattice\).  The image of the subset of \(\group\)-invariant
valuations generates a cosimplicial convex cone \(\valCone\) called the \emph{valuation cone} of \(\homo\).

The \(\group\)-equivariant embeddings of \(\homo\) are in one-to-one correspondence with \emph{colored fans} (see \cite{Kno91} for a detailed exposition of this correspondence).  
Let \(\variety\) be a complete \(\group\)-equivariant embedding of \(\homo\), with colored fan \(\coloredFan\). 
Let \(\divBinv\) be the set of \(\borel\)-stable prime divisors in \(\variety\). 
It is a finite set comprised of closures of codimension one \(\borel\)-orbits in \(\homo\) and of closures of codimension one \(\group\)-orbits, the latter corresponding to colorless rays in \(\coloredFan\). 
We identify such a divisor \(D\) with the induced valuation and thus get an associated element \(\colorMap(D)\) of \(\realDualLattice\) via the map \(\colorMap\). 

Let \(\lineBundle\) be an ample \(\group\)-linearized line bundle on \(\variety\). 
Then \(\group\) acts on the space of holomorphic sections of~\(\lineBundle\). 
Fix a \(\borel\)-eigenvalue \(\globalSection\) for this action, and denote its \(\borel\)-weight by \(\weightSection\). 
The \(\borel\)-invariant Cartier divisor on \(\variety\) defined by \(\globalSection\) is of the form 
\begin{equation}
\label{Cartier}
\sum_{D \in \divBinv} n_D D,   
\end{equation}
and there exists an integral piecewise linear function \(f\) defined on \(\realDualLattice\) such that \((f \circ \colorMap)(D) = n_D\) whenever \(D\) contains a \(\group\)-orbit (in other words, when either \(D\) is \(\group\)-invariant, or \(D\) is a \emph{color} of \(\variety\)).  
Conversely, a divisor as in~\eqref{Cartier} with the same property defines a Cartier divisor. 

The condition that \(\lineBundle\) is ample is equivalent to the condition that the minimal function \(f\) which satisfies the above condition further satisfies the following:
\begin{enumerate}[label=\roman*)]
\item \(f\) is a convex function, 
\item \((f \circ \colorMap)(D) < n_D\) for each \(D\in\divBinv\) which does not contain \(\group\)-orbits, 
\item the slopes of \(f\) on two distinct maximal cones of \(\coloredFan\) are distinct. 
\end{enumerate}

Brion defines a polytope \(\polytope = \polytope\polVar\) associated to \(\globalSection\) as the convex polytope in \(\realSphericalLattice\) defined by the equations \(\colorMap(D)(m) + n_D \geq 0\) for all \(D \in \divBinv\).
Integral points of \(\polytope\) give the decomposition of \(H^0\polVar\) as a \(\group\)-representation. 
More precisely, if \(V_{\lambda}\) denotes an irreducible representation of \(\group\) with highest weight \(\lambda\), then \(H^0\polVar\) is \(\group\)-isomorphic to 
\begin{equation*}
\bigoplus_{m\in \sphericalLattice \cap \polytope} V_{\weightSection + m}, 
\end{equation*}
where we recall that \(\weightSection\) is the \(\borel\)-weight of \(\globalSection\). 

The degree \(\lineBundle^n\) is obtained from this polytope by Brion as follows.  Let \(\rootSystem^+\) denote the positive root system of \((\group,\borel,\torus)\), and let \(\relevantRoots\) denote the set of positive roots that are \emph{not} orthogonal to \(\weightSection + \polytope\).  Then
\begin{equation}
\label{polytope_volume}
\lineBundle^n = n! \int_{\polytope} \DHpol \lebesgue,
\end{equation}
where \(\lebesgue\) is the Lebesgue measure on \(\realSphericalLattice\) normalized by \(\sphericalLattice\), and \(\DHpol\) is the Duistermaat--Heckman polynomial defined by
\begin{equation}
\label{eqn_DHpol}
\DHpol(x) = \prod_{\alpha \in \relevantRoots} \frac{\langle x + \weightSection , \alpha \rangle}{\langle \halfSum , \alpha \rangle}, 
\end{equation} 
where \(\halfSum\) is the half sum of positive roots of \(\group\).  The above result is proved by considering the first-order asymptotic of the dimensions of the spaces of pluri-sections \(H^0(\variety,\lineBundle^k)\) and Weyl's dimension formula
\begin{equation*}
\dim(V_{\lambda}) = \prod_{\alpha \in \rootSystem^+} \frac{\langle \alpha, \lambda + \halfSum \rangle}{\langle \alpha, \halfSum \rangle}.
\end{equation*}
Indeed, up to the \(n!\) factor, the highest-order (in \(k\)) coefficient of \(\sum_{m\in \sphericalLattice \cap k \polytope} \dim(V_{k \weightSection + m})\) gives the integral in~\eqref{polytope_volume}, and the polynomial appearing in this volume is the highest-order summand of the polynomial giving the dimension formula, restricted to the affine space \(\weightSection + \realSphericalLattice\).

We will actually need the following refinement, which is a consequence of a general result of Pukhlikov and Khovanski\u{\i} \cite{PK92a}. The dimension \(\dim H^0(\variety,\lineBundle^k)\) admits an expansion in powers of \(k\) of the form
\begin{equation}
\label{Pukhlikov-Khovanskii}
 \dim H^0(\variety,\lineBundle^k) = k^{n} \int_{\polytope} \DHpol \lebesgue + k^{n - 1} \left( \frac{1}{2} \int_{\partial \polytope} \DHpol \lebesgueBoundary + \int_{\polytope} \coDHpol \lebesgue \right) + o(k^{n-1}),
\end{equation}
where 
\begin{equation}
\label{eqn_coDHpol}
\coDHpol(x) = \sum_{\alpha \in \relevantRoots} \frac{\langle \alpha, \halfSum \rangle}{\langle \alpha, x + \weightSection \rangle} P(x)
\end{equation}
and \(\lebesgueBoundary\) denotes the measure on \(\partial \polytope\) which coincides on each facet with the Lebesgue measure normalized by the intersection of \(\sphericalLattice\) with the affine space spanned by the face.  We explain how this result follows from \cite{PK92a} in the next subsection.

\subsection{Second coefficient in the expansion}

The proof of expansion~\eqref{Pukhlikov-Khovanskii} arises from a generalization of Ehrhart's multiplicity theorem, as follows from work of McMullen \cite{McM77} and Khovanski\u{\i}--Pukhlikov \cite{PK92a}. 

Let \(V\) be a real vector space and \(\Lambda\) a lattice in \(V\). 

Khovanski\u{\i} and Pukhlikov's general results  (in particular, \cite[Corollary 2.5]{PK92a}) show that, on the group of virtual polytopes, which consists essentially of formal finite real linear combinations of polytopes equipped with Minkowki addition, the evaluation of a given fixed homogeneous polynomial at integral points of a polytope with vertices in \(\Lambda\) extends to a polynomial function (valuation in the terminology of \cite{PK92a}) from the group of virtual polytopes to \(\bbR\). 
In particular, its restriction to the subgroup generated by a single polytope is a polynomial function. 

Assume that \(f\colon V\to \bbR\) is a homogeneous polynomial of degree \(d\) and \(\hat{\polytope}\) is an \(r\)-dimensional polytope in \(V\) with vertices in \(\Lambda\); then it follows from the results quickly summarized above that the function
\[ k\in\bbZ_{>0} \longmapsto \sum_{x\in \Lambda\cap k\hat{\polytope}} f(x) \] 
is the restriction of a polynomial function \(\mathcal{F}\) from \(\bbR\) to \(\bbR\). 
We note that 
\[ \mathcal{F}(k) = k^d \sum_{x\in \frac{1}{k}\Lambda\cap \hat{\polytope}} f(x) = k^{d+r}\sum_{x\in \frac{1}{k}\Lambda\cap \hat{\polytope}} \frac{f(x)}{k^r}, \]
and the second sum converges to \(\int_{\hat{\polytope}} f \lebesgue\) for the Lebesgue measure \(\lebesgue\) on the affine span of \(\hat{\polytope}\), normalized by the lattice; hence the polynomial \(\mathcal{F}\) is a degree \(d+r\) polynomial whose degree \(d+r\) coefficient is \(\int_{\hat{\polytope}} f \lebesgue\).

The formula for the second coefficient of the polynomial \(\mathcal{F}\) follows from Minkowski's inversion in the group of virtual polytopes \cite[Theorem 2.2]{PK92a}: the inverse of the polytope \(\hat{\polytope}\) in this group is \((-1)^{r}\Int(-\hat{\polytope})\) (which can be interpreted as the formal sum \(\sum_{F}(-1)^{\dim(F)}(-F)\), where \(F\) runs over all faces of \(\hat{\polytope}\) and \(-F\) denotes the polytope symmetric to \(F\) with respect to the origin in \(V\)).  It thus follows from the polynomiality theorem that the polynomial function \(\mathcal{F}\) admits the following expression at negative integers:
\[\mathcal{F}(-k) = (-1)^{r}\sum_{x\in \Lambda\cap \Int(-k\hat{\polytope})} f(x) 
= (-k)^{r+d}\sum_{x\in \frac{1}{k}\Lambda \cap \Int(\hat{\polytope})} \frac{f(x)}{k^r}. \]

Write \(\mathcal{F}(k)=\sum_{j=0}^{r+d} a_{j}k^{j}\).  We already know that \(a_{r+d}=\int_{\hat{\polytope}} f \lebesgue\), and we want to find the second coefficient.  Consider the sum
\[ \mathcal{F}(k)-(-1)^{r+d}\mathcal{F}(-k) = \sum_{j=0}^{r+d} \left(a_{j}-(-1)^{r+d+j}a_j\right)k^{j}; \]
then of course its highest possibly non-zero coefficient is the coefficient of \(k^{r+d-1}\). 
On the other hand, in  view of previous formulas, the left-hand side may be interpreted as 
\[ \mathcal{F}(k)-(-1)^{r+d}\mathcal{F}(-k) = k^{r+d-1}\sum_{x\in \frac{1}{k}\Lambda \cap \partial \hat{\polytope}} \frac{f(x)}{k^{r-1}}. \]
We deduce from the two equalities that 
\[2a_{r+d-1} = \lim_{k\to \infty} \frac{\mathcal{F}(k)-(-1)^{r+d}\mathcal{F}(-k)}{k^{r+d-1}} = \int_{\partial \hat{\polytope}} f \mathop{d\sigma},\]
where \(\mathop{d\sigma}\) is the measure on the boundary which, restricted to a facet, coincides with the translate of the Lebesgue measure on the linear space spanned by the facet, normalized by the induced lattice.

To get the expansion~\eqref{Pukhlikov-Khovanskii}, we apply this to the two highest-degree homogeneous components in the Weyl dimension formula.  One should however be careful in choosing the right space to apply this.  Consider the vector space \(V:= (M\otimes \bbR) \oplus \bbR\), equipped with the lattice \(\Lambda:= \{(m,k)\mid m-k\chi \in M\}\), where \(\chi\) is the \(B\)-weight of a fixed \(B\)-eigensection \(s\) of \(L\).  Consider the \(r\)-dimensional polytope
\[ \hat{\polytope}=\{(x,1)\mid x\in \polytope \}, \]
where \(\polytope\) is the polytope associated to \((X,L)\) as in Section~\ref{sec31} 
Consider the degree \(d\) polynomial \(f\) extending the Weyl dimension formula applied to the weights \(x+t\chi\) for \(x\in M\)  and \(t\in \bbZ_{\geq 0}\), 
\[ f\colon V\lra \bbR,\quad (x,t)\longmapsto \prod_{\alpha\in \rootSystem} \frac{\langle\alpha,x+t\chi+\varpi \rangle}{\langle \alpha, \varpi\rangle}, \]
and let \(f_d\) and \(f_{d-1}\) denote its highest-degree homogeneous components. 
Note that \(P(x)=f_d(x,1)\) and \(Q(x)=f_{d-1}(x,1)\). 
Applying the results of \cite{PK92a} and the argument to compute the second coefficient described above yields the expansion~\eqref{Pukhlikov-Khovanskii}.

\section{Test configurations for polarized spherical varieties}
\label{sec_test_config}

\subsection{Statement}
In this section, we encode equivariant test configurations for polarized spherical varieties by certain concave piecewise linear functions.  In addition to Donaldson's work on toric varieties \cite{Don02}, this task has already been accomplished in different special cases; see \cite{AB04,Nyb,DelKSSV}.  We freely use notation from Section~\ref{sec_spherical}.

\begin{thm}
\label{thm_TC}
For a polarized spherical variety \(\polVar\), \(\group\)-equivariant test configurations are in one-to-one correspondence with positive rational piecewise linear concave functions on \(\polytope\polVar\), with slopes in the valuation cone of\, \(\variety\).

The test configuration is furthermore special if the associated function is integral linear, and it is a product test configuration if it is integral linear, with slope in \(\linValCone\).

Furthermore, rational twists of a given test configuration \(\polVarTC\) are in one-to-one correspondence with elements of\, \(\linValCone\cap\rationalDualLattice\), and if \(g\) is the function associated to \(\polVarTC\), then the set of functions corresponding to the twists is \(\{g + l \mid l \in \linValCone\cap\rationalDualLattice\}\).
\end{thm} 

The correspondence is explicitly described in the proof below. 
The key picture to keep in mind is that the polytope associated (as in \cite{Bri89}) to a trivially compactified test configuration can be described as the set of point below the graph of a concave integral piecewise linear function on some multiple of the polytope associated to the initial polarized variety.  

\subsection{From a test configuration to a concave function\ldots}

Let \(\polVar\) be a polarized spherical variety.  Let \(\polVarTC\) be a \(\group\)-equivariant test configuration for \(\polVar\).  We still denote by \(\polVarTC\) the \emph{trivially compactified} \(\group\)-equivariant test configuration for \(\polVar\).  That is, we glue the trivial family over \(\bbC\) to \(\polVarTC\) along \(\bbC^*\) to obtain a family over \(\bbP^1\).  We denote the point added to \(\bbC\) by \(\infty\) and keep the notation \(\polVarTC\) for the family over \(\bbP^1\).

Note that \(\polVarTC\) is a polarized spherical variety under the action of \(\group\times\bbC^*\) (this remark will also be used in Odaka's appendix).  Its open orbit is \(\homo\times\bbC^*\), and the combinatorial data is easily derived from that of~\(\variety\).

Let \(\globalSectionTC\) be the \(\bbC^*\)-invariant meromorphic section of \(\lineBundleTC\) whose restriction to \((\variety,\lineBundle^r) = (\fiberTC_1, \lineBundleTC_1)\) coincides with~\(\globalSection^{\otimes r}\).  The divisor associated to \(\globalSectionTC\) is \((\borel \times \bbC^*)\)-stable, hence an integral linear combination of the form
\begin{equation}
\label{divisorTC}
\sum_{\hat{D} \in \divBinvTC} n_{\hat{D}} \hat{D}, 
\end{equation}
where \(\divBinvTC\) is the set of prime \((\borel\times\bbC^*)\)-stable divisors on \(\varietyTC\). 

There are three types of such divisors: 
\begin{itemize}
\item Each divisor \(\hat{D} \in \divBinvTC\) with \(\colorMapTC(\hat{D}) \in (\realDualLattice \times \{0\})\) must be of the form \(\hat{D}= \overline{D \times \bbC^*}\) for some \(D \in \divBinv\).  Then by our choice of section, \(n_{\hat{D}} = r n_D\).  Note that since the Borel subgroup of \(\bbC^*\) is \(\bbC^*\) itself, all other elements of \(\divBinvTC\) must be \((\group\times\bbC^*)\)-stable.
\item There is only one divisor \(\hat{D}\in \divBinvTC\) such that \(\colorMapTC(\hat{D}) \in (\valCone \times \bbR_{> 0})\); this is the fiber \(\fiberTC_{\infty}\) because \(\polVarTC\) is trivial at \(\infty\).  
The choice of section, on the other hand, implies that for this divisor, \(n_{\hat{D}} = 0\). 
\item Let \(\mathcal{A}\) denote the set \(\colorMapTC(\divBinvTC) \cap (\valCone \times \bbR_{< 0})\), and write each element in \(\mathcal{A}\) as an ordered pair \((u,s)\) with \(u\in\valCone\) and \(s\in \bbR_{< 0}\). 
Let us also write \(n_{u,s}\) for the corresponding coefficient in~\eqref{divisorTC}.
\end{itemize}

Let \(\polytopeTC\) denote the polytope in \(\sphericalLatticeTC \otimes \bbR\) associated to the divisor~\eqref{divisorTC}. 
In view of the previous description of the divisor, the polytope can be described as 
\[ \polytopeTC = \{ (r x,t) \mid x \in \polytope, 0 \leq t \leq g(x) \}, \]
where \(g\) is a (positive) rational concave piecewise linear function on \(\polytope\), expressed as 
\[ g(x) = \inf_{(u,s) \in \mathcal{A}} \left( \frac{r u(x) + n_{u,s}}{-s} \right). \] 
Note that since each \(s\) is negative, the slopes \((-s)^{-1}ru\) are rational points of the valuation cone \(\valCone\). 

\begin{rem}
It may not seem natural that the positive direction corresponds to \(\infty\) for the reader accustomed to a certain point of view on toric varieties. 
It stems from the fact that, under the action \(w\cdot f(z) = f(w^{-1}z)\), the function \(f\colon z\mapsto z^k\) is a \(\bbC^*\)-eigenvector with eigenvalue the one-parameter subgroup \(\chi\colon w \mapsto w^{-k}\) rather than \(f\) itself. 
\end{rem}

\subsection{\ldots and back}

We now explain how to reconstruct a test configuration from a concave function.
Let \(g\) be a positive rational piecewise linear concave function on \(\polytope\polVar\), with slopes in the valuation cone of \(\variety\). 
We can find a positive integer \(r\), a subset \(\mathcal{A}\) of \(\valCone \times \bbZ_{< 0}\), and integers \(n_{u,s}\) for all \((u,s) \in \mathcal{A}\) such that 
\[ g(x) = \inf_{(u,s) \in \mathcal{A}} \left( \frac{r u(x) + n_{u,s}}{-s} \right) \]
and \(u\) is a primitive element of \(\dualLattice\) for all \((u,s)\in\mathcal{A}\). 

Consider the polytope 
\[ \polytopeTC = \{ (r x,t) \mid x \in \polytope, 0 \leq t \leq g(p) \}. \]
We build a colored fan \(\coloredFanTC\) for the \((\group\times\bbC^*)\)-homogeneous space \(\homo\times\bbC^*\) from \(\polytopeTC\) as follows. 
Recall that we are given \(\polVar\) and the corresponding divisor~\eqref{Cartier}, and that colors of \(\homo\times\bbC^*\) may be identified with colors of \(\homo\). 
We first include in \(\coloredFanTC\) the colored cones \((\sigma \times \{0\},S)\) and \((\sigma \times \bbR_{> 0},S)\), where \((\sigma,S)\) is a colored cone of \(\coloredFan\). 
These account for the trivial family over \(\bbC^*\cup\{\infty\}\). 
To complete the fan, we add, for each cone \(\sigma\) in the opposite of the normal fan to \(\polytopeTC\) %
which has not been considered yet and whose intersection with the interior of \(\valConeTC\) is non-empty, a colored cone \((\sigma,S) \in \coloredFanTC\), where \(S\) is defined as follows. 
It suffices to define it for maximal colored cones. 
For such a cone \(\sigma\), let \(m_{\sigma}\) denote the corresponding vertex of \(\polytopeTC\). 
Then \(S\) is the set of all colors \(D\) of \(\homo\times\bbC^*\) in \(\sigma\) such that \(-\colorMap(\overline{D})(m_{\sigma}) + n_D > 0\). 

We have thus defined a colored fan, hence an embedding \(\varietyTC\) of \(\homo\times\bbC^*\). 
As follows from the description of equivariant morphisms between spherical varieties, see \cite[Theorem~4.1]{Kno91}, 
\(\varietyTC\) admits a \(\bbC^*\)-equivariant surjective morphism to \(\bbP^1\), which induces a trivial family with fiber \(\variety\) over the affine chart \(\bbC^* \cup \{\infty\}\). 
We identify this subvariety with \(\variety\times(\bbC^*\cup\{\infty\})\).

The polytope \(\polytopeTC\) is the polytope associated to the \((\borel\times\bbC^*)\)-stable Cartier divisor 
\[ d = \sum_{\hat{D}\in \divBinv} n_{\hat{D}} \hat{D},\]
where 
\begin{itemize}
\item \(n_{\hat{D}} = r n_D\) for each divisor \(\hat{D} \in \divBinvTC\) with \(\colorMapTC(\hat{D}) \in (\realDualLattice \times \{0\})\) (equivalently \(\hat{D}= \overline{D \times \bbC^*}\) for some \(D \in \divBinv\));   
\item \(n_{\hat{D}} = 0\) for the only divisor \(\hat{D}\in \divBinvTC\) such that \(\colorMapTC(\hat{D}) \in (\valCone \times \bbR_{> 0})\), which is the fiber \(\fiberTC_{\infty}\); 
\item \(n_{\hat{D}} = n_{u,s}\) for \((u,s)\in \mathcal{A}\) and \(\hat{D}\) the \(\group\)-stable divisor which is the closure of the codimension one \(\group\)-orbit associated to the colorless ray generated by \((u,s)\). 
\end{itemize}
In particular, the restriction of this divisor to \(\variety\times(\bbC^*\cup\{\infty\})\) is the product of the divisor~\eqref{Cartier} with \(\bbC^*\cup\{\infty\}\). 
Furthermore, this divisor satisfies the ampleness assumption. 
This is not obvious since the colors of \(\varietyTC\) are different from the colors of \(\variety\) in general, but our choices of colors for each colored cone were tailored to ensure ampleness. 
The associated line bundle \(\mathcal{O}(d)\) is \((\group\times\bbC^*)\)-linearizable (maybe up to passing to a suitable finite tensor power, which does not seriously affect our statement). 
Choosing the linearization such that the natural section \(\globalSection\) of \(\mathcal{O}(d)\) is \(\bbC^*\)-invariant and has \(\borel\)-weight \(\weightSection\)
yields the final identification of \(\mathcal{O}(d)\) with the pullback of \(\lineBundle\) by the first projection on \(\variety\times(\bbC^*\cup\{\infty\})\).
This concludes the construction of the test configuration \(\polVarTC\). 

\subsection{Effect of twisting}

We now elucidate the different possible twists of a given test configuration, as involved in the definition of \(\group\)-uniform K-stability. 
For a spherical homogeneous space \(\homoTC\), one can easily identify \(\Aut_{\group\times\bbC^*}(\homoTC)\): it is the group \(N_{\group\times\bbC^*}(\isotropy\times\{1\})/(\isotropy\times\{1\})\), acting on the right on \(\homoTC\). 
Furthermore, this group is diagonalizable, and the action of its neutral component \(F\) extends to any embedding. 
Finally, \(\mathfrak{Y}(F)\times \bbR\) may be identified with the linear part \(\linValCone\) of the valuation cone of \(\homoTC\). 

The above is not actually necessary since we can identify the possible twists directly by the theory of spherical embeddings. 
Indeed, the colored fans of the twist of a test configuration and of the initial test configurations are the same, as is the combinatorial data identifying the line bundle. 
Note that what we just wrote is true for the total space of the test configuration itself but not for the compactification, which depends on the twist. 
The only difference is thus that the privileged direction coming from the factor \(\bbC^*\) can be chosen differently, and that will affect the final expression of \(g\). 

More precisely, the direction can be chosen arbitrarily among those directions in \((\linValCone \cap \dualLattice) \times \{1\}\), or \((\linValCone \cap (\rationalDualLattice)) \times \{1\}\) to allow for rational twists. 
The effect on \(g\) is by adding the function \(l\) for some element \(l\) of \(\linValCone \cap \rationalDualLattice\). 
This concludes the proof of Theorem~\ref{thm_TC}.

\section{Non-Archimedean functionals for spherical test configurations}
\label{sec_NA}


In this section, we compute the non-Archimedean Mabuchi functional and the non-Archimedean \(J\)-functional of the test configuration \(\polVarTC\) associated to the concave function \(g\). 
The computation follows the method of \cite{Don02} and was previously used to obtain sub-cases of our result in \cite{Nyb,AK05}.

We use the notation of Section~\ref{sec_spherical}, and we set  
\begin{equation*}
V := \int_{\polytope} \DHpol \lebesgue
\quad\text{and}\quad \scalar := \frac{1}{2 V} \left( \int_{\partial \polytope} \DHpol \lebesgueBoundary + 2 \int_{\polytope} \coDHpol \lebesgue \right).
\end{equation*}

\begin{thm}
\label{thm_NA}
We have 
\[ \MNA\polVarTC = \frac{1}{V}\left( \scalar \int_{\polytope} g \DHpol \lebesgue - \frac{1}{2} \int_{\partial \polytope} g \DHpol \lebesgueBoundary - \int_{\polytope} g \coDHpol \lebesgue \right) \] 
and
\[ \JNA\polVarTC = \frac{1}{V} \int_{\polytope} \left(\max_{\polytope} g - g\right) \DHpol \lebesgue . \]
\end{thm}
With the notation of the introduction, we thus have \(2 V \MNA\polVarTC = \LFunctional(g)\) and \(V \JNA\polVarTC = \JFunctional(g)\). 
We provide an expression for \(\MNA\) rather than for the Donaldson--Futaki invariant. 
This is because it is linear with respect to base change, so it is enough to compute it up to base change. 
We can thus in particular assume that the test configuration is reduced and use the definition of the Donaldson--Futaki invariant recalled in Section~\ref{sec_stability}. 
This reduction will appear during the proof. 

\begin{proof}
Let \(\fiberTC_x\) denote the fiber of the test configuration for \(x \in \bbP^1 = \bbC \cup \{\infty\}\). 
Donaldson uses the following exact (for large enough \(k\)) sequences of \(\bbC^*\)-representations obtained by restriction of sections:
\[ 0 \lra H^0(\varietyTC,\lineBundleTC^k \otimes \mathcal{O}(-\fiberTC_{0})) \lra H^0(\varietyTC,\lineBundleTC^k) \lra H^0(\fiberTC_{0}, \lineBundleTC^k|_{\fiberTC_{0}}) \lra 0,  \]
\[ 0 \lra H^0(\varietyTC,\lineBundleTC^k \otimes \mathcal{O}(-\fiberTC_{\infty})) \lra H^0(\varietyTC,\lineBundleTC^k) \lra H^0(\fiberTC_{\infty}, \lineBundleTC^k|_{\fiberTC_{\infty}}) \lra 0.  \]
Donaldson further notes that the \(\bbC^*\)-action on \(H^0(\fiberTC_{\infty}, \lineBundleTC^k|_{\fiberTC_{\infty}})\) is trivial and that the family of weights of the \(\bbC^*\)-representation \(H^0(\varietyTC,\lineBundleTC^k\otimes \mathcal{O}(-\fiberTC_{\infty}))\) is \((\lambda_i + 1)_{i\in I}\) if \((\lambda_i)_{i\in I}\) is the family of weights of the \(\bbC^*\)-representation \(H^0(\varietyTC,\lineBundleTC^k\otimes \mathcal{O}(-\fiberTC_{0}))\).   

This allows us to express the quantities involved in the definitions of the Donaldson--Futaki invariant and the non-Archimedean \(J\)-functional as follows. 
The sum \(w_k\) of weights of \(H^0(\fiberTC_{0}, \lineBundleTC^k|_{\fiberTC_{0}})\) is given by 
\[ w_k = \dim H^0(\varietyTC,\lineBundleTC^k) - \dim H^0(\fiberTC_{0}, \lineBundleTC^k|_{\fiberTC_{0}}). \]
Furthermore, if \(\lambda_k\) denotes the maximum of all weights of \(H^0(\fiberTC_{0}, \lineBundleTC^k|_{\fiberTC_{0}})\), then 
\[\sup_k \frac{\lambda_k}{k} = \max_{\polytope} g. \] 

In order to use the expansion~\eqref{Pukhlikov-Khovanskii} applied to \(\polytopeTC\), we may remark that for any (say continuous) function \(f\) on \(\polytope\),
\begin{equation*} 
\int_{\polytopeTC} f(x) \lebesgue(x,t) = \int_{\polytope}  f(x) g(x) \lebesgue(p) 
\end{equation*}
and 
\begin{equation*}
\int_{\partial \polytopeTC} f(x) \lebesgueBoundary(x,t) = 2 \int_{\polytope}  f(x) \lebesgue(x) + \int_{\partial \polytope} g(x) f(x) \lebesgueBoundary(x).  
\end{equation*}
The latter equality follows from the decomposition of the boundary as the slice \(\polytopeTC\cap(\realSphericalLattice\times\{0\})\), the graph \(\{(x,g(x))\mid p\in \polytope\}\) of \(g\) (giving each one half of the first summand), and the vertical part \(\{(x,t)\mid x\in \partial\polytope,\) \mbox{\(0\leq t\leq g(x)\}\)} giving the second summand.
For the graph of \(g\) to give the right contribution, given the definition of \(\lebesgueBoundary\), it is actually \emph{necessary to assume} that \(g\) is defined by \emph{integral} linear forms. We can restrict to this case by base change since we are interested in the non-Archimedean Mabuchi functional rather than the Donaldson--Futaki invariant. 

By the expansion~\eqref{Pukhlikov-Khovanskii} applied to both \(\polytope\) and \(\polytopeTC\), we obtain the following expansions: 
\begin{equation*}
k d_k = k^{n+1} \int_{\polytope} \DHpol \lebesgue + k^n \left( \frac{1}{2} \int_{\partial \polytope} \DHpol \lebesgueBoundary + \int_{\polytope} \coDHpol \lebesgue \right) + o(k^n),
\end{equation*}
\begin{equation*}
w_k = k^{n+1} \int_{\polytope} g \DHpol \lebesgue + k^n \left( \frac{1}{2} \int_{\partial \polytope} g \DHpol \lebesgueBoundary + \int_{\polytope} g \coDHpol \lebesgue \right) + o(k^n).
\end{equation*}
Writing \(w_k = A k^{n+1} + B k^n + o(k^n)\) and \(k d_k = C k^{n+1} + D k^n + o(k^n)\), we have 
\begin{equation*}
\frac{w_k}{k d_k} = \frac{A}{C} + \frac{1}{C}\left(B - \frac{A D}{C}\right) \frac{1}{k} + o\left(\frac{1}{k}\right).  
\end{equation*} 
Substituting the expressions above proves Theorem~\ref{thm_NA}. 
\end{proof}

Combining the results of Sections~\ref{sec_test_config} and~\ref{sec_NA} proves Theorem~\ref{thm_Kstab}.

\section{Restating the problem}
\label{sec_new}

In this section, we will show how Theorem~\ref{thm_sufficient} applies to the uniform K-stability problem. 
For this, we will obtain a new expression of \(\LFunctional\) when applied to smooth functions and show how to derive uniform K-stability in these terms. 
To simplify the notation, we assume (by choosing the global section appropriately) that the origin \(0\in \realSphericalLattice\) is in the interior of the polytope \(\polytope\). 

\subsection{A new expression of \(\boldsymbol{\LFunctional}\) on smooth functions}

Let \(E_1, \ldots, E_k\) denote the facets of \(\polytope\), and let \(T_i\) denote the pyramid with vertex the origin and base \(E_i\).
This provides in particular a decomposition \(\polytope = \bigcup_{i} T_i\). 
The author learned the idea of using such a decomposition in \cite{ZZ08,LZZ18}.

We will need notation for the set of equations defining \(\polytope\): for each facet, let \(u_i\) denote the outward-pointing primitive normal in \(\dualLattice\), 
and let \(n_i\) be the positive number such that 
\[ \polytope = \{ x \in \realSphericalLattice \mid \forall i, u_i(x) \leq n_i \}. \]

Let \(J\) and \(K\) be the functions on \(\polytope\) defined (almost-everywhere) by 
\begin{equation}
\label{eqn_J}
J(x) = \frac{- \DHpol(x)}{n_i} 
\end{equation} 
and 
\begin{equation}
\label{eqn_K}
K(x) = 2 \scalar \DHpol(x) - 2 \coDHpol(x) - \frac{1}{n_i} d_x\DHpol(x) - \frac{1}{n_i} r \DHpol(x) 
\end{equation} 
for \(x \in \Int(T_i)\), where \(r\) denotes the dimension of \(\realSphericalLattice\), also called the \emph{rank} of \(\variety\). 

Note that these functions are not continuous in general but piecewise polynomial with respect to the decomposition of the polytope, hence integrable. Furthermore, \(J\) is negative on the interior of \(\polytope\). 

\begin{prop}
\label{prop_smooth}
For any continuous function \(f\) on \(\polytope\), smooth on the interior, we have 
\[ \LFunctional(f) = \int_{\polytope} (f(x) K(x) + d_xf(x) J(x)) \lebesgue(x). \] 
\end{prop}

\begin{proof}
We identify \(\realSphericalLattice\) with the Euclidean space \(\bbR^r\) by choosing a basis of \(\sphericalLattice\).  
Let \(\nu\) denote the unit outward-pointing normal vector to \(\partial\polytope\), and let \(\mathop{d\sigma_e}\) denote the area measure on \(\partial\polytope\). Also, let \((x \cdot \nu)\) denote the scalar product of \(x \in \realSphericalLattice\) with \(\nu\) induced by the identification with \(\bbR^r\). 

For \(f\) smooth, the divergence theorem yields, for all \(i\),
\begin{equation*}
\int_{E_i} f(x) \DHpol(x)  (x \cdot \nu) \mathop{d \sigma_e}(x) = 
\int_{T_i} \left( f(x) d_x\DHpol(x) + d_xf(x) \DHpol(x) + r f(x) \DHpol(x) \right) \lebesgue(x).  
\end{equation*}
Note that the considered vector field  is radial; hence there are no contributions from the other facets of \(T_i\). 

Let \(c_i\) denote the constant such that \(u_i(x)=c_i (x \cdot \nu|_{E_i})\) for \(x\in \realSphericalLattice\). 
Then \(\mathop{d\sigma_e} = c_i\lebesgueBoundary\) on \(E_i\) and \((x \cdot \nu) = \frac{n_i}{c_i}\) for \(x\in E_i\); hence on the left-hand side above, we may replace \((x\cdot\nu) \mathop{d \sigma_e}\) with \(n_i \lebesgueBoundary\).

Then using the decomposition \(\polytope = \bigcup_{i} T_i\) to rewrite the boundary term in \(\LFunctional\), we have 
\begin{equation*}
\LFunctional(f) = \sum_i \int_{T_i} (K(x) f(x) + J(x) d_xf(x)) \lebesgue(x)
\end{equation*}
by the definitions of \(K\) and \(J\).
\end{proof}

\begin{rem}
\label{rem_expression_JK}
In the case when the restriction of \(\DHpol\) to the facet \(E_i\) vanishes, we can replace the value of \(n_i\) in the expressions of \(J|_{T_i}\) and \(K|_{T_i}\) with any number or even with \(+\infty\), in the sense that one can take \(J=0\) and \(K=2\scalar \DHpol - 2\coDHpol\) on \(T_i\). 
\end{rem}

\subsection{Working on smooth functions}

As in the introduction, we choose a complement \(\complementLinValCone\) of \(\linValCone\) in \(\realDualLattice\), and we denote by \(\smoothNormalized\) the space of continuous concave functions on \(\polytope\), smooth in the interior with differentials in \(\valCone\), such that \(\max f = 0\) and \(d_0f \in \complementLinValCone\). 

To replace the boundary integral used in the toric case by Donaldson, we introduce
\[ \LFunctional_+(f) := \int_{\polytope} (K_+(x)f(x)+J(x)d_xf(x))\lebesgue(x), \]
where, for \(x\in \Int(T_i)\), 
\[ K_+(x) = \sup(2\scalar P(x) - 2Q(x), 0) - \frac{1}{n_i}d_xP(x)-\frac{1}{n_i}rP(x). \]
                                          
\begin{lem}
\label{lem_positive_contrib}
There exists a constant \(\eta>0\) such that for all \(f\in \smoothNormalized\), 
\[ \LFunctional_+(f) \geq \eta \int_{\polytope} (-f)P\lebesgue \]
\end{lem}

\begin{proof}	
  First note that, as in the proof of Proposition~\ref{prop_smooth}, 
\[ \LFunctional_+(f) = \int_{\polytope} 2f\sup(2\scalar P(x) - 2Q(x), 0)\lebesgue - \int_{\partial\polytope} f P d\sigma. \]

We first claim that if \(x\in \polytope\) and \(P(x)=0\), then there exists a neighborhood \(V\) of \(x\) such that we have \((2\scalar P - 2Q)(y) <0\) for \(y\in V\cap \polytope\). 
Indeed, let \(\mathcal{S}\) be the subset of all \(\alpha\in \relevantRoots\) such that \(\langle \alpha, x + \chi \rangle=0\). 
Then the dominant term in \(\scalar P-Q\), for \(y\) near \(x\), is 
\[ - \sum_{\alpha \in \mathcal{S}} \frac{\langle \alpha, \varpi \rangle}{\langle \alpha, y + \chi \rangle} P(y), \]
which is negative for \(y\in \polytope\) near \(x\). 

Let us work in polar coordinates. 
For a given direction \(\theta \in \mathbb{S}^{r-1}\), consider the ray \(t\theta\), and let \(s\) denote the maximal positive real number such that \(s\theta\in \polytope\). 
If \(P(s\theta)=0\), then the concavity and normalization, plus the previous observation, show that for a neighborhood \(V'\) of \(\theta\) in \(\mathbb{S}^{r-1}\) and \(\polytope'=\{t\theta \in \polytope \mid \theta \in  V'\}\), there exists a \(\eta'>0\) such that for all \(f\in \smoothNormalized\), 
\[ \int_{\polytope'} 2f\sup(2\scalar P(x) - 2Q(x), 0)\lebesgue \geq \eta' \int_{\polytope'} (-f)P\lebesgue. \]
If \(P(s\theta)>0\), then there is a compact neighborhood \(F\) of \(s\theta\) in \(\partial\polytope\) such that \(P>0\) on \(F\). Then setting \(\polytope'=\{ty \mid t\in [0,1], y \in F\}\), by convexity and normalization, there exists a \(\eta'>0\) such that for all \(f\in \smoothNormalized\), 
\[ \int_{F} -f P d\sigma \lebesgue \geq \eta' \int_{\polytope'} (-f)P\lebesgue. \]
By the compactness of \(\partial\polytope\), we obtain the result. 
\end{proof}

\begin{prop}
\label{prop_from_to}
The polarized \(\group\)-spherical variety \(\polVar\) is \(\group\)-uniformly K-stable if \(\LFunctional\) vanishes on elements of \(\linValCone\) and there exists an \(\varepsilon > 0\) such that for all \(f \in \smoothNormalized\), 
\begin{equation}
\label{eqn_unif_normalized}
\LFunctional(f) \geq \varepsilon \LFunctional_+(f). 
\end{equation}
It is K-semistable if \(\LFunctional\) is invariant under addition of an element of\, \(\linValCone\) and \(\LFunctional \geq 0\) on \(\smoothNormalized\). 
\end{prop}

\begin{proof}
Assume that \(\LFunctional\) is invariant under addition of an element of \(\linValCone\) and there exists an \(\varepsilon > 0\) such that \eqref{eqn_unif_normalized} holds for all \(f \in \smoothNormalized\). 

Let \(\polVarTC\) be a \(\group\)-equivariant test configuration for \(\polVar\), and let \(g\) denote the associated positive concave rational piecewise linear function, with slopes in \(\valCone\). 
The first step is to note that \(g\) can be uniformly approximated on \(\polytope\) by a sequence \((f_m)\) of smooth concave functions with slopes in \(\valCone\). 

For each \(m\), set 
\[ \hat{f}_m := f_m - p_{\linValCone}(d_0f_m) - \max_{\polytope} (f_m - p_{\linValCone}(d_0f_m)), \] 
where \(p_{\linValCone}\) is the linear projection on \(\linValCone\) relative to \(\complementLinValCone\). 
Then \(\hat{f}_m\in \smoothNormalized\), and since the \(f_m\) are uniformly Lipschitz, it subconverges uniformly to a function \(g-l -\max_{\polytope}(g - l)\) for some \(l\in \linValCone\).  

By assumption~\eqref{eqn_unif_normalized}, Lemma~\ref{lem_positive_contrib}, and uniform convergence, we have 
\[ \LFunctional\left(g-l -\max_{\polytope}(g - l)\right) \geq \varepsilon \int_{\polytope} \left(\max_{\polytope}(g-l) - g + l\right) \DHpol \lebesgue. \]
By the invariance of \(\LFunctional\) under addition of a constant, or an element of \(\linValCone\), we can replace the left-hand side with \(\LFunctional(g)\). 
We finally have 
\[ \LFunctional(g) \geq \varepsilon \int_{\polytope} \left(\max_{\polytope}(g-l) - g + l\right) \DHpol \lebesgue \geq  \varepsilon \inf_{l' \in \linValCone} \JFunctional(g + l'). \]
We have proved uniform K-stability by Theorem~\ref{thm_Kstab}. 

For K-semistability, it suffices to follow the same arguments with \(\varepsilon = 0\).  
\end{proof} 

\section{A combinatorial sufficient condition}
\label{sec_proof_sufficient}

In this section, we will prove Theorem~\ref{thm_sufficient}. 
The proof is rather elementary and follows from a well-chosen decomposition of \(\LFunctional\) as a sum of terms which are each non-negative under the assumptions. 
We will begin with a simpler analogue of Theorem~\ref{thm_sufficient} that uses this decomposition and then proceed to the proof.

\subsection{A condition for semistability}
\label{sec_sufficient_ss}

For now, let \(b\) be any point in \(\polytope\).
Write the integrand of the functional \(\LFunctional_s\) as 
\begin{align}
f(x)K(x) + d_xf(x)J(x) = & \left(d_xf(x-b) - f(x) + f(b)\right)J(x) \label{eqn_concave1} \\
& \quad + \left(f(x) - f(b) - d_bf(x-b)\right) (K + J)(x) \label{eqn_concave2} \\ 
& \quad + d_xf(b) J(x) \label{eqn_barycenter_condition} \\
& \quad + f(b) K(x) \label{eqn_constant} \\
& \quad + d_bf(x-b) (K + J)(x). \label{eqn_barycenter_definition} 
\end{align}
We thus have a decomposition of \(\LFunctional_s\) as a sum of each corresponding integral.  

\begin{prop}
\label{thm_sufficient_ss}
Assume that \(K + J \leq 0\).
Let \(b\) be the element of \(\polytope\) defined by 
\begin{equation*}
\int_{\polytope} (x-b) (K(x)+J(x)) \lebesgue(x) = 0.
\end{equation*}
If\, \(b\) is in \(-\valCone^{\vee}\), then 
\[ \LFunctional_s(f) \geq 0 \]
for any continuous concave function \(f\) on \(\polytope\), smooth in the interior, with differentials in \(\valCone\). 
\end{prop}

\begin{proof}
Using \(b\) as defined in the statement, we consider the decomposition of \(\LFunctional_s\) as above. Then 
\begin{itemize}
\item the contribution from summands~\eqref{eqn_concave1}~and~\eqref{eqn_concave2} is non-negative by the concavity and non-positivity of \(J\) and \(K+J\). 
\item the contribution~\eqref{eqn_barycenter_condition} is non-negative by the assumption on the barycenter \(b\) and the fact that the differentials of \(f\) are in \(\valCone\), 
\item the contribution of the summand~\eqref{eqn_constant} is zero since the integral of \(K\) is zero,  
\item and the contribution~\eqref{eqn_barycenter_definition} is zero by the definition of \(b\). \hfill\qed
\end{itemize}
\renewcommand{\qed}{}
\end{proof}

\subsection{A preparatory pre-compactness result}

For the full proof of Theorem~\ref{thm_sufficient} and for future reference, we will use the following pre-compactness result, which is a generalization of one used by Donaldson \cite[Corollary~5.2.5]{Don02}. 

\begin{prop}
\label{prop_compactness}
Let \(C\) be a positive real number. 
Any sequence of non-positive concave functions \((f_m)\) on \(\Delta\) with \(\int_{\Delta}(-f) P \lebesgue \leq C\) has a sub-sequence which converges to a concave function \(f_{\infty}\) on the interior of \(\Delta\), and the convergence is uniform over strict compact subsets of \(\Delta\). 
\end{prop}

For the proof, we use the standard Euclidean structure on \(\bbR^n\).   
For any \(x\in\Delta\), let \(d_x\) denote the distance from \(x\) to the boundary \(\partial\Delta\).  
For any positive number \(d\), set 
\[ \Delta_d := \{ x \in \Delta \mid d_x \geq d \}. \] 
Note that by the continuity of \(P\), for any positive \(d\), there exists a positive constant \(\delta_d\) such that \(P \geq \delta_d\) on~\(\Delta_d\). 
Recall that if \(f\) is a concave function on \(\polytope\), a linear function \(l\) is a \emph{superdifferential} of \(f\) at \(x\) if for all \(y\), \(f(y) \leq f(x) + l(y-x)\). 

\begin{lem}\label{lem73}
For any $($small enough$)$ positive \(d\), there exists a positive constant \(\kappa = \kappa(d)\) such that for any point \(x \in \Delta_d\), for any non-positive concave function \(f\) on \(\Delta\) with finite \(\int_{\Delta} f P \lebesgue\), and for any superdifferential \(l\) of \(f\) at~\(x\), 
\[ \lVert\, l\, \rVert \leq \kappa \int_{\Delta} (-f) P \lebesgue. \]
\end{lem}

\begin{proof}
Consider the ball \(B\) of center \(x\) and radius \(d/2\), which is contained in \(\Delta_{d/2}\), and the half-ball \(B^-\subset B\) where the affine function \(y\mapsto l(y-x)\) is negative. 
Then 
\[ \int_{B^-} l(y-x) P(y) \lebesgue(y) \leq - \lVert\, l\, \rVert \delta_{d/2} \left(\frac{d}{2}\right)^{r+1} C, \]  
where \(C\) is a positive constant independent of \(x\), \(d\), \(f\). 
We have furthermore 
\begin{align*}  
\int_{B^-} l(y-x) P(y) \lebesgue(y) & \geq \int_{B^-} (l(y-x) + f(x)) P(y) \lebesgue(y) \quad \text{since }f\text{ is non-positive} \\
& \geq \int_{B^-} f(y) P(y) \lebesgue(y) \quad \text{by the definition of a superdifferential} \\
& \geq \int_{\Delta} f(y) P(y) \lebesgue(y) \quad \text{by the non-positivity again.}
\end{align*}
This concludes the proof.
\end{proof}

\begin{proof}[Proof of Proposition~\ref{prop_compactness}]
By concavity, we have 
\[ 0 \geq \max_{\Delta_d} f \geq \frac{1}{\int_{\Delta_d} P \lebesgue}\int_{\Delta_d} f P \lebesgue \geq \frac{1}{\int_{\Delta_d} P \lebesgue}\int_{\Delta} f P \lebesgue, \]
and we deduce from Lemma~\ref{lem73}  
\[ \min_{\Delta_d} f \geq \max_{\Delta_d} f + \mathrm{diam}(\Delta) \kappa(d) \int_{\Delta} f P \lebesgue. \]
As a consequence, a bound on \(\int_{\Delta} f P \lebesgue\) provides a uniform bound and a Lipschitz bound on any strict compact subset of \(\Delta\). 
The proposition then follows by applying the Arzel\`a--Ascoli theorem. 
\end{proof}

\subsection{End of the proof of Theorem~\ref{thm_sufficient}}\label{sec73}

By Proposition~\ref{prop_smooth}, we have \(\LFunctional_s=\LFunctional\) for continuous functions that are smooth on the interior of \(\polytope\). 
In particular, this is true for any linear function \(l\in \realDualLattice\). 
Furthermore, for \(l\in \realDualLattice\), we have 
\[\LFunctional(l)=\LFunctional_s(l)=\int_{\polytope} l (K+J) \lebesgue = -C_3 b\]
for some positive constant \(C_3\).

As a consequence, the condition that \(b\) is in the relative interior of \( -\valCone^{\vee}\) is equivalent to \(\LFunctional(l)\geq 0\) for all \(l\in\valCone\), with equality if and only if \(l\in \linValCone\). 
In particular, it implies that \(\LFunctional\) vanishes on \(\linValCone\). 

We prove the result by contradiction using Proposition~\ref{prop_from_to}. Let \((f_m)\) be a sequence of functions in \(\smoothNormalized\). Assume for the sake of a contradiction that \(\LFunctional(f_m) \to 0\) while \(\LFunctional_+(f_m)=1\) for all \(m\). 
The second condition implies, by Lemma~\ref{lem_positive_contrib} and Proposition~\ref{prop_compactness}, the existence of a limit \(f_{\infty}\) defined on the interior of the polytope, such that \((f_m)\) converges to \(f_{\infty}\) uniformly on every compact subset of the interior of \(\polytope\). 

Furthermore, by using the expression of \(\LFunctional(f_m)\) as a sum of non-negative terms as in Section~\ref{sec_sufficient_ss}, we obtain that each individual term converges to zero. 
In particular, that 
\[ \lim_{m\to\infty} f_{m}(x) - f_{m}(b) - d_bf_{m}(x-b) = 0\] 
almost-everywhere implies that \(f_{\infty}\) is affine. 
The slope of \(f_{\infty}\) is necessarily in \(\complementLinValCone \cap \valCone\) by the normalization. 
Finally, the assumption on the barycenter \(b\) and the term~\eqref{eqn_barycenter_condition} in Section~\ref{sec_sufficient_ss} imply that the slope must be zero. 
Hence \(f_{\infty}=0\). 
Back to our assumptions, we have \(\lim_{m\to\infty} \LFunctional(f_m)-\LFunctional_+(f_m) = -1\), but on the other hand, 
\(\LFunctional(f_m)-\LFunctional_+(f_m) = \int_{\polytope} \inf(2\scalar P(x)-Q(x),0) f_m(x) \lebesgue(x)\) converges to zero by the local uniform convergence of \(f_m\) to the zero function. 
We have reached a contradiction. 
\qed

\section{Full statement}
\label{sec_full}

Let us wrap up in this section the statement of the sufficient condition for uniform K-stability of polarized spherical varieties we proved. 

Let us note that \(\polytope\) is not the most direct choice of polytope associated to the polarized spherical variety \(\polVar\). 
The \emph{moment polytope} \(\polytope_+\) is in some sense more natural to consider as it depends on fewer choices. 
Recall that the moment polytope of \(\polVar\) is the convex polytope obtained by taking the closure of the set of all (normalized) \(\borel\)-weights of plurisections of \(\lineBundle\). 
It does not lie in the same space as \(\polytope\) in general. 
More precisely, the relation between the two is a simple translation: \(\polytope_+ = \weightSection + \polytope\), and \(\polytope_+\) lies in the affine space \(\weightSection + \realSphericalLattice\). 
One easily sees from the previous sections that it is not important in our results for \(\weightSection\) to be in~\(\sphericalLattice\). 
Hence, the data of \(\polytope_+\) alone allows one to recover both \(\polytope\) and one choice (or several choices) of \(\weightSection\) to apply our sufficient criterion. 
On the other hand, the data of \(\sphericalLattice\) is not readily read off from \(\polytope_+\) alone. 
The importance of \(\sphericalLattice\) in the statement is seen through the integers \(n_i\). 

The full statement for our sufficient condition for \(\group\)-uniform K-stability is as follows.

\begin{thm}\label{thm81}
Let \(\polVar\) be a polarized \(\group\)-spherical variety with spherical lattice \(\sphericalLattice\) of rank \(r\), valuation cone \(\valCone\), and moment polytope \(\polytope_+\). Let \(\relevantRoots\) denote the positive roots of \(\group\) not orthogonal to \(\polytope_+\). 
Choose an element  \(\weightSection\) in the interior of \(\polytope_+\), number the facets of the translated polytope \(\polytope := -\weightSection+\polytope_+\) by \(E_i\) for \(1\leq i\leq s\), and let \(n_i\) be the positive numbers such that 
\[ \polytope = \{ x \in \realSphericalLattice \mid u_i(x) \leq n_i \}, \]
where \(u_i\in \dualLattice=\mathrm{Hom}(\sphericalLattice,\bbZ)\) denotes the outward-pointing primitive normal to \(E_i\). 
For \(m=r\) or \(r+1\), let \(L_m\) be the almost-everywhere defined function on \(\polytope\) such that for \(x\) in the interior of the convex hull of \(E_i\cup\{0\}\), 
\[ L_m(x) = \sum_{\alpha\in\relevantRoots} \left\langle \alpha, \left(m n_i^{-1} - 2 \scalar \right)(x+\weightSection) + \mathrm{Card}\left(\relevantRoots\right)\left( n_i^{-1} x + 2 \halfSum \right)  \right\rangle \prod_{\beta\in \relevantRoots\setminus \{\alpha\}} \langle \beta, x+\weightSection \rangle, \]
where the constant \(\scalar\) is defined by \(\int_{\polytope} L_r \lebesgue = 0\) for some Lebesgue measure \(\lebesgue\) on \(\realSphericalLattice\). 

Assume that \(L_{r+1}\) is non-negative on \(\polytope\), and let \(b\) denote the barycenter of \(\polytope\) with respect to the measure \(L_{r+1} \lebesgue\). 
Assume furthermore that \(-b\) is in the relative interior of the valuation cone \(\valCone\). 
Then \(\polVar\) is \(\group\)-uniformly K-stable.  

If\, \(L_{r+1}\) is non-negative on \(\polytope\) and \(b\in -\valCone^{\vee}\), then \(\polVar\) is \(\group\)-equivariantly K-semistable.
\end{thm}

\begin{proof}
We work under the assumptions of the theorem; that is, \(L_{r+1}\) is non-negative on \(\polytope\), and the barycenter \(b\) of \(\polytope\) with respect to the measure \(L_{r+1} \lebesgue\) is in the relative interior of \(-\valCone\). 
Note that \(K = -C_1 L_r\) and \(K+J = -C_2 L_{r+1}\) for some positive constants \(C_1\) and \(C_2\). 
In particular, the barycenters \(b\) involved in Theorem~\ref{thm_sufficient} and the theorem we are trying to prove are indeed the same. 
Theorem~\ref{thm_sufficient} thus provides the uniform K-stability result, while the K-semistability result follows from Proposition~\ref{thm_sufficient_ss}.
\end{proof}

\begin{rem}
\label{rem_expression_L}
It follows from the relation between \(L_m\) and \(J\) and \(K\) in the proof and Remark~\ref{rem_expression_JK} that the statement of Theorem~\ref{thm81} applies as well if we replace \((n_i)^{-1}\) with any non-negative number for the facets \(E_i\) where \(\DHpol\) vanishes.  
\end{rem}

Applying Corollary~\ref{sph.YTD} from Odaka's appendix to this paper, we have the following sufficient condition for the existence of cscK metrics.

\begin{cor}
Assume that \(L_{r+1}\) is non-negative, that \(b\) is in the relative interior of \(-\valCone^{\vee}\) and that \(\variety\) is smooth. Then there exists a cscK metric in \(c_1(\lineBundle)\).
\end{cor}

On the other hand, from the point of view of K-stability alone, our theorem can be interpreted in the following way. 
\begin{cor}
Under the assumption that \(L_{r+1}\) is non-negative on \(\polytope\), \(\group\)-uniform K-stability is equivalent to \(\group\)-stc K-polystability for \(\polVar\). 
\end{cor}
In particular, this point of view shows that our barycenter condition is in fact necessary. 

\begin{proof}
As noted at the end of the proof of Theorem~\ref{thm_sufficient}, in Section~\ref{sec73}, the condition that \(b\) is in the relative interior of \( -\valCone^{\vee}\) is equivalent to \(\LFunctional(l)\geq 0\) for all \(l\in\valCone\), with equality if and only if \(l\in \linValCone\). 
By Theorem~\ref{thm_TC}, this is equivalent to \(\group\)-stc K-polystability. 
\end{proof}

\begin{rem}[On the toric case]
In the toric case, the statement simplifies a lot; let us state it anew. 
Let \((X,L)\) be a polarized toric variety with integral moment polytope \(\polytope\subset \bbR^r\), such that \(0\) is in the interior of \(\polytope\). 
Number the facets of the polytope \(\Delta\) by \(E_i\) for \(1\leq i\leq s\), and let \(n_i\) be the positive numbers such that 
\[ \polytope = \{ x \in \bbR^r \mid u_i\cdot x \leq n_i \}, \]
where \(u_i\in \bbZ^r\) denotes the outward-pointing primitive normal to \(E_i\). 
In this case, we have \(L_m\equiv mn_i^{-1}-2\scalar\) in the interior of the convex hull of \(E_i\cup\{0\}\),   where \(\scalar \) is such that \(\int_{\polytope} L_r \lebesgue = 0\). 
The condition thus becomes: assume that \((r+1)n_i^{-1}-2\scalar\leq 0\) for all \(i\) and that the barycenter of \(\polytope\) with respect to the measure \(L_{r+1} \lebesgue\) is zero. 
If these conditions are satisfied, then \(\polVar\) is \(\group\)-uniformly K-stable.  

This statement may be new in the singular and uniform K-stability setting, but actually the whole proof in this toric situation is essentially contained in Zhou and Zhu's arguments to prove a sufficient condition for coercivity of the Mabuchi functional in \cite{ZZ08}. 
\end{rem}

\section{Example: Blowup of \texorpdfstring{\(\boldsymbol{\mathcal{Q}^3}\)}{Q\textasciicircum 3} along  \texorpdfstring{\(\boldsymbol{\mathcal{Q}^1}}{Q\textasciicircum 1}\)}
\label{sec_example}

In this section, we study the blowup \(X\) of the three-dimensional quadric along a one-dimensional subquadric. This example was previously considered in \cite{DH} and \cite{infinite}, where it was presented in more details. 
The Picard rank of this variety is two. 

The connected reductive group \(G\) making \(X\) a rank two spherical variety is \(\SL_2\times \bbC^*\). 
We fix a choice of maximal torus and Borel subgroup. 
Let \(\alpha\) denote the unique positive root, and let \(f\) denote the character of weight \(1\) of \(\bbC^*\). 
The spherical lattice \(\sphericalLattice\) is the lattice generated by \(\alpha\) and \(\frac{1}{2}(\alpha + f)\). 
The dual lattice \(\dualLattice\) is the lattice generated by \(\alpha^{\vee}\) and \(\frac{1}{2}(\alpha^{\vee}+f^{\vee})\), where \(\alpha^{\vee}\) is the coroot of \(\alpha\) and \(f^{\vee}\) is defined similarly by \(f^{\vee}(f)=2\) and \(f^{\vee}(\alpha)=0\). 
The valuation cone \(\mathcal{V}\) is the dual cone to \(\bbR_-\alpha\). 
Finally, we have \(\relevantRoots=\{\alpha\}\) and \(\halfSum=\alpha\). 

The moment polytope for an ample line bundle on \(X\) is, up to scaling, of the following form for some \(s>3/2\):  
\[\polytope(s):= \{x \alpha + y f \in \realSphericalLattice \mid 0\leq x\leq 3/2,\;\; x-s \leq y \leq s-x \}. \]
Its four facets \(E_0,\ldots,E_3\) have respective outward-pointing primitive normals \(u_0=-\alpha^{\vee}\), \(u_1=\frac{1}{2}(\alpha^{\vee}+f^{\vee})\),
\(u_2=\alpha^{\vee}\) and \(u_3=\frac{1}{2}(\alpha^{\vee}-f^{\vee})\). 

\begin{figure}
\centering
\caption{The polytope $\Delta(s)$ for \(s=2\)}
\begin{tikzpicture}  
\draw [dotted] (-3,-3) grid[xstep=1,ystep=1] (3,3);
\draw (0,0) node{$\bullet$};
\draw (0,0) node[above left]{$0$};
\draw [thick] (-2,-2) -- (1,-2) -- (2,-1) -- (2,2) -- cycle;
\draw (1.5,-1.5) node{$\bullet$};
\draw (1.5,-1.5) node[below right]{$\frac{3}{2}\alpha$};
\draw (2,2) node{$\bullet$};
\draw (2,2) node[above right]{$s f$};
\end{tikzpicture}
\end{figure}
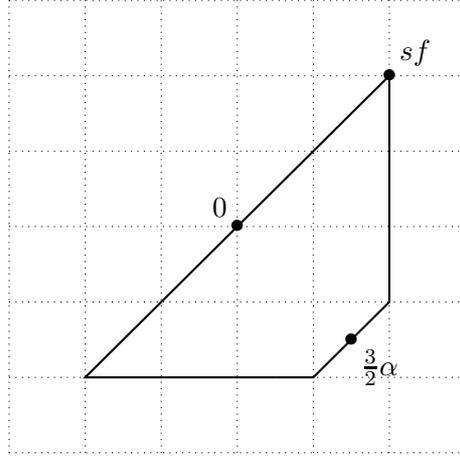

Before trying to apply our theorem, let us compute the important quantity \(\scalar\). For this, it is more convenient to deal with the boundary integral
\[ 2\scalar = \frac{\int_{\polytope}2\coDHpol\lebesgue + \int_{\partial\polytope}\DHpol\lebesgueBoundary}{\int_{\polytope}\DHpol\lebesgue}. \]
We have \(\DHpol(x\alpha+yf)=2x\), \(\coDHpol\equiv 1\), \(\lebesgue = 2 \mathop{dx}\mathop{dy}\) in the coordinates \(x\alpha+yf\), \(\lebesgueBoundary|_{E_1}=\lebesgueBoundary|_{E_3} = 2\mathop{dx}\) if we parametrize by \(x\) and \(\lebesgueBoundary|_{E_0}=\lebesgueBoundary|_{E_2} = \mathop{dy}\). 
We thus have 
\begin{align*} 
\int_{\polytope} 2\coDHpol\lebesgue & = \int_{x=0}^{\frac{3}{2}} \int_{y=x-s}^{s-x} 4 \mathop{dy}\mathop{dx} = 12 s - 9, \\
\int_{\polytope}\DHpol\lebesgue & = \int_{x=0}^{\frac{3}{2}} \int_{y=x-s}^{s-x} 4x \mathop{dy}\mathop{dx} = 9(s-1), \\
\int_{E_0} \DHpol \lebesgueBoundary & = 0, \\
\int_{E_1} \DHpol \lebesgueBoundary = \int_{E_3} \DHpol \lebesgueBoundary & = \int_{0}^{\frac{3}{2}} 4x\mathop{dx} = \frac{9}{2}, \\
\int_{E_2} \DHpol \lebesgueBoundary & = \int_{\frac{3}{2}-s}^{s-\frac{3}{2}} 3 \mathop{dy} = 6s-9, 
\end{align*}
and 
\[ 2\scalar = \frac{2s-1}{s-1}. \]

By similar computations, we can check whether the barycenter condition involved in our theorem, or equivalently, the \(\group\)-stc K-polystability, holds.
This amounts to the two conditions \(\LFunctional(f^{\vee})=0\) and \(\LFunctional(-\alpha^{\vee})>0\). 
The first of these conditions is automatic by the symmetry of the moment polytopes and Duistermaat--Heckman polynomial. 
We compute the second, using the expression with a boundary integral as for \(2\scalar\); we obtain  
\[ \LFunctional(-\alpha^{\vee})=\frac{9(8s^2-18s+11)}{4(s-1)}, \]
which is positive for any \(s>3/2\). 

We now choose an element of the interior of the polytope (and this is the tricky part to get the theorem to apply). 
For reasons related by the general Fano case to be treated next, we choose \(\weightSection = \frac{1}{2}s \alpha\), which can be considered of course only if \(s<3\). 
Then the translated polytope \(\polytope = -\weightSection + \polytope_+\) is defined by the four equations \(u_i(x \alpha + y f)\leq n_i\) with 
\begin{align*}
n_0  = s, \quad n_1 = n_3  = \frac{s}{2}, \quad n_2  = 3 - s. 
\end{align*}

\begin{figure}
\centering
\caption{The decomposition of $\Delta(s)$}
\begin{tikzpicture}  
\draw (1,-1) node{$\bullet$};
\draw [thick] (-2,-2) -- (1,-2) -- (2,-1) -- (2,2) -- cycle;
\draw (0,0) node[above left]{$E_0$};
\draw (2,0) node[right]{$E_1$};
\draw (2,-2) node{$E_2$};
\draw (0,-2) node[below]{$E_3$};
\draw [dotted] (1,-1) -- (-2,-2);
\draw [dotted] (1,-1) -- (1,-2);
\draw [dotted] (1,-1) -- (2,-1);
\draw [dotted] (1,-1) -- (2,2);
\end{tikzpicture}
\end{figure}
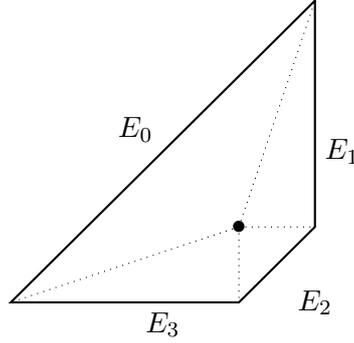

In order to show that for these polarizations, \(\group\)-uniform K-stability is equivalent to the barycenter condition, 
we have to check that for all \(i\) and for all \(x\alpha+yf\in T_i\), we have 
\[ \left\langle \alpha, \left(3n_i^{-1}-2\scalar\right)\left(\left(x+\frac{s}{2}\right)\alpha+yf\right)+\left(n_i^{-1}x+1\right)\alpha+n_i^{-1}yf \right\rangle \geq 0. \]
Recall that \(\langle \alpha, f\rangle = 0\); hence it suffices to check 
\begin{align*}
(4n_0^{-1}-2\scalar)x + (3n_0^{-1}-2\scalar)\frac{s}{2}+1 \geq 0 & \quad\text{ for } -\frac{s}{2} \leq x \leq 0, \\
(4n_1^{-1}-2\scalar)x + (3n_1^{-1}-2\scalar)\frac{s}{2}+1 \geq 0 & \quad\text{ for } -\frac{s}{2} \leq x \leq \frac{3-s}{2}, \\
(4n_2^{-1}-2\scalar)x + (3n_2^{-1}-2\scalar)\frac{s}{2}+1 \geq 0 & \quad\text{ for } 0\leq x \leq \frac{3-s}{2}, \\
(4n_3^{-1}-2\scalar)x + (3n_3^{-1}-2\scalar)\frac{s}{2}+1 \geq 0 & \quad\text{ for } -\frac{s}{2} \leq x \leq \frac{3-s}{2}. 
\end{align*}
Since \(n_1=n_3\), the second and fourth conditions are equivalent. 
Since \(\DHpol\) vanishes on \(E_0\), we can choose any \(n_0\) we want, for example \(n_0=n_1\), so that the first condition is implied by the second. 
We end up with only two conditions to check.

Replacing \(n_1\) and \(2\scalar\) with their expression in \(s\), one of the conditions is 
\[ \frac{2s(-2s^2+9s-8)(2x+s)}{s-1} \geq 0 \quad\text{ for } -\frac{s}{2} \leq x \leq \frac{3-s}{2}. \]
The degree two polynomial \(-2s^2+9s-8\) is non-negative for \(\frac{9-\sqrt{17}}{4} \leq s \leq \frac{9+\sqrt{17}}{4}\), which contains the range \(\frac{3}{2} < s < 3\), and the other factors are easily checked to be non-negative for the values of \(x\) and \(s\) considered, so the condition is satisfied. 

The other condition, replacing \(n_2\) and \(2\scalar\) with their expression in \(s\), is 
\[ \frac{2s^2-3s-1}{(s-1)(3-s)} x + \frac{s^3-3s^2+4s-3}{(3-s)(s-1)} \geq 0 \quad\text{ for } 0 \leq x \leq \frac{3-s}{2}. \]
Since \(s>\frac{3}{2}>1\), one easily checks that the coefficient of \(x\) is positive. 
Hence the condition is satisfied if 
\[ 0 \geq -\frac{s^3-3s^2+4s-3}{2s^2-3s-1}. \]
Again, for \(s>1\) as here, this is equivalent to 
\[ s^3-3s^2+4s-3 \geq 0, \]
and one can check that the unique real root \(s_0\) of this cubic polynomial is approximately \(s_0 \approx 1.6823\), in the range of \(s\) considered. 

To sum up, we have shown that \(\group\)-uniform K-stability of the polarized variety is equivalent to the barycenter condition if \(s\) is such that \(s_0 \approx 1.6823 \leq s < 3\). 
We have thus proved that \(X\) admits a cscK when \(s_0 \leq s < 3\). 
It is very likely that our choice of \(\weightSection\) was not the optimal one and that one can push further the use of our main theorem to get a slightly larger range of classes with cscK metrics. 
The natural question regarding this example is whether it is a Calabi dream space in the sense of Chen and Cheng, that is, if all classes admit cscK metrics. 
It is unlikely that our main theorem is enough for this, but we intend to answer this question in a later work by studying optimal degenerations for rank two spherical threefolds. 

\section{Fano case}
\label{sec_Fano}

In this section, we apply our main theorem to the case of a Fano spherical manifold equipped with its anticanonical polarization. 

\begin{thm}
\label{thm_Fano}
Assume that \(\variety\) is \(\bbQ\)-Fano and that \(\lineBundle\) is $($a multiple of\;$)$ the anticanonical \(\bbQ\)-line bundle of \(\variety\). Then \(L_{r+1}\) is positive on \(\polytope\). 
\end{thm}

This shows, with a very straightforward proof, that for spherical Fano varieties, \(\group\)-uniform K-stability is equivalent to \(\group\)-stc K-polystability and allows one to recover the explicit combinatorial condition for this obtained in \cite{DelKSSV}. 

\begin{proof}
Assume for simplicity that \(\variety\) is Gorenstein and \(\lineBundle=K_X^{-1}\). Apart from notational issues, the general case is the same. Assume furthermore that \(K_X^{-1}\) is equipped with its canonical \(\group\)-linearization. 

The proposition will follow from the judicious choice \(\weightSection = 2\halfSum_X = \sum_{\alpha\in\relevantRoots}\alpha\) and two steps: 
\begin{enumerate}
\item Using Remark~\ref{rem_expression_L}, we can assume that \(n_i=1\) for all \(i\) in the expression of \(L_m\). 
\item We can replace \(2\halfSum\) with \(2\halfSum_X\) in the expression of \(L_m\).  
\end{enumerate}

Let us begin with the simple case of toric varieties. 
We obviously have \(\halfSum = \halfSum_X = 0\) in the toric case. 
The torus-invariant section of \(K_X^{-1}\) has weight zero, which corresponds to the unique interior integral point of the moment polytope \(\polytope_+\). 
Finally, the polytope \(\polytope_+\) is defined by equations \(u_i(x)\leq 1\) for a set of primitive elements \(u_i\) in \(\dualLattice\); hence by choosing \(\weightSection=0\), we have \(n_i=1\) for all \(i\). 
Actually, an integral polytope is reflexive if and only if it is defined by equations \(u_i(x)\leq 1\) for a set of primitive elements \(u_i\) in \(\dualLattice\), and it is well known that Gorenstein Fano toric varieties correspond to reflexive polytopes. 

For the general case, we will use the description of the anticanonical divisor of spherical varieties by Brion in \cite{Bri87} as formulated in \cite{GH15Fano}.  
Namely, there exists a section of the anticanonical line bundle, with \(\borel\)-weight \(2\halfSum_X\), whose divisor  \(\sum_{D\in\divBinv} n_D D\) is such that \(n_D = 1\) if \(D\) is \(\group\)-stable, and the description of the coefficients of colors (closures of \(\borel\)-stable divisors in \(\homo\)) is explicit, depending on the type of each color. 
Since we need the details, let us quickly recall the possible types of colors. 

Let \(S\) denote the set of simple roots of \(\group\). 
For \(\alpha\in S\), let \(P_{\alpha}\) denote the largest parabolic in \(\group\) containing \(\borel\) such that \(-\alpha\) is not a root of \(P_{\alpha}\). 
Let \(\mathcal{D}(\alpha)\) denote the set of \(\borel\)-stable prime divisors of \(X\) that are not \(P_{\alpha}\)-stable. 
It turns out that these divisors exhaust the set of divisors in \(\divBinv\) that are not \(\group\)-stable and that \(\mathcal{D}(\alpha)\) is non-empty precisely if \(\alpha\in\relevantRoots\). 
An element \(D\in \mathcal{D}(\alpha)\) is 
\begin{itemize}
\item of type \(a\) if \(\alpha\) is a primitive element of \(\sphericalLattice\), 
\item of type \(2a\) if \(2\alpha\) is a primitive element of \(\sphericalLattice\),
\item and of type \(b\) otherwise. 
\end{itemize}
The coefficient \(n_D\) is then obtained, depending on the type of \(D\in \mathcal{D}(\alpha)\), as follows:
\begin{itemize}
\item \(n_D = \frac{1}{2} \alpha^{\vee}(2\halfSum_X) = 1\) for type \(a\) or \(2a\), 
\item \(n_D = \alpha^{\vee}(2\halfSum_X)\) for type \(b\),
\end{itemize}
where \(\alpha^{\vee}\) denotes the coroot of \(\alpha\).  

Now consider the polytope \(\polytope\) associated to the section constructed by Brion. 
We want to check that we can take \(n_i = 1\) for each facet. 
Recall that the equations defining \(\polytope\) from the coefficients of the divisor are the \(\colorMap(D)(x)+n_D\geq 0\) for \(D\in\divBinv\). 
In particular, the equations defining the facets are of the form \(\colorMap(D)(x)+n_D = 0\). 
To put these in the form \(u_i(x) = n_i\), where \(u_i\) is a primitive outer normal, one has to find the positive number \(a_i\) such that \(-\colorMap(D)/a_i\) is primitive, in which case one can take \(n_i = n_D/a_i\).

Assume that \(D\in \divBinv\) defines a facet of \(\polytope\).

Whenever \(D\) is \(\group\)-stable, \(\colorMap(D)\) is primitive and \(n_D=1\), so we have \(n_i=1\). 

Next, assume that \(D\) is a color of type \(a\), and let \(\alpha\in S\) be such that \(D\in \mathcal{D}(\alpha)\). 
Then  \(\alpha(\colorMap(D)) = 1\) (see \cite{GH15homo} for a convenient summary of the properties of colors by type), which implies that \(\colorMap(D)\) is primitive. 
Since \(n_D=1\) as well in this case, we indeed have \(n_i=1\). 

Now, assume that \(D\) is a color of type \(2a\), and let \(\alpha\in S\) be such that \(D\in \mathcal{D}(\alpha)\). 
Then \(\colorMap(D) = \frac{1}{2} \alpha^{\vee}|_{\sphericalLattice}\). 
If \(x\) is in the facet defined by \(\colorMap(D)(x)+n_D = 0\), then we have \(\frac{1}{2}\alpha^{\vee}(x)+\frac{1}{2}\alpha^{\vee}( 2\halfSum_X) = 0\), which implies \(\langle \alpha, x+2\halfSum_X \rangle = 0\). 
Since \(\weightSection=2\halfSum_X\), this in turn implies that \(\DHpol\) vanishes on the facet. 
By Remark~\ref{rem_expression_L}, we can then choose \(n_i=1\) in the expressions of \(L_r\) and \(L_{r+1}\). 

Similarly, if \(D\) is a color of type \(b\) and \(\alpha\in S\) is such that \(D\in \mathcal{D}(\alpha)\), we have \(\colorMap(D) = \alpha^{\vee}|_{\sphericalLattice}\) and \(n_D = \alpha^{\vee}(2\halfSum_X)\); hence \(\DHpol\) must vanish on the facet. 
Again by Remark~\ref{rem_expression_L}, we can choose \(n_i=1\) in the expressions of \(L_r\) and \(L_{r+1}\). 

We now turn to the problem of replacing \(\halfSum\) by \(\halfSum_X\) in the expression of \(L_m\). 
The important property of \(\relevantRoots\) is that it consists of all roots of the unipotent radical of some parabolic subgroup \(P_X\) of \(\group\), namely the stabilizer of the open \(\borel\)-orbit in \(\variety\). 
As a consequence, we may write \(\halfSum_X=\halfSum-\halfSum_0\), where \(\halfSum_0\) is the sum of positive roots of the Levi subgroup of \(P_X\). 
Let \(W_0\) denote the Weyl group of this Levi subgroup, which is a subgroup of the Weyl group of \(\group\). 
The action of \(W_0\) on roots of \(\group\) induces a permutation of \(\relevantRoots\). 
Consider the linear function 
\[h\colon y \mapsto \sum_{\alpha\in\relevantRoots} \langle \alpha, w\cdot y \rangle \prod_{\beta\in\relevantRoots\setminus\{\alpha\}}\langle \beta, x+\weightSection\rangle. \]
The discussion above shows that \(h\) is invariant under the action of \(W_0\). 
In particular, \(h(\halfSum_0) = 0\) since there exists a \(w_0 \in W_0\) such that \(w_0(\halfSum_0) = -\halfSum_0\). 
We thus have 
\begin{align*}
L_m(x) & = h\left(\left(mn_i^{-1}-2\scalar\right)(x+\weightSection)+\mathrm{Card}\left(\relevantRoots\right)\left(n_i^{-1}x+2\halfSum\right)\right) \\
& = h\left(\left(mn_i^{-1}-2\scalar\right)(x+\weightSection)+\mathrm{Card}\left(\relevantRoots\right)\left(n_i^{-1}x+2\halfSum_X\right)\right).
\end{align*}

Putting together all ingredients in the case \(\lineBundle=K_X^{-1}\) (\(\weightSection=2\halfSum_X\), all \(n_i=1\), and we can replace \(2\halfSum\) with \(2\halfSum_X\)), we have 
\begin{align*} 
L_m(x) & = h\left(\left(m-2\scalar+\mathrm{Card}\left(\relevantRoots\right)\right)(x+\weightSection)\right) \\
& =  \left(m-2\scalar+\mathrm{Card}\left(\relevantRoots\right)\right) \prod_{\alpha\in\relevantRoots} \langle \alpha, x+\weightSection \rangle.  
\end{align*}
Since \(\prod_{\alpha\in\relevantRoots} \langle \alpha, x+\weightSection \rangle \) is positive on the interior of \(\polytope\) and \(\int_{\polytope} L_r \lebesgue = 0\), we deduce that \(2\scalar = r + \mathrm{Card}(\relevantRoots)\). 
This number actually coincides with the dimension of the variety and could be recovered by interpreting directly \(2\scalar\) as the average scalar curvature of the Fano variety \(X\). 

Finally, 
\[ L_{r+1}(x) = \prod_{\alpha\in\relevantRoots} \langle \alpha, x+\weightSection \rangle \]
is strictly positive on the interior of \(\polytope\) and non-negative on the whole polytope. 
\end{proof}

\section{Polarizations close to the anticanonical line bundle}
\label{sec_open}

In this final section, we illustrate how our main theorem applies to different situations to give equivalence of \(\group\)-uniform K-stability and \(\group\)-stc K-polystability for polarizations close to the anticanonical line bundle. 
We however believe that the criterion better shows its strength when applied to a concrete situation as in Section~\ref{sec_example}. 

\begin{prop}
\label{prop_open}
Let \(\polVar\) be a \(\group\)-spherical variety, and assume that there exist a choice of \(\weightSection\in \polytope_+\) and a positive number \(\delta\) such that \(L_{r+1}>\delta\) on \(\polytope\). 
Then for polarizations close to \(\lineBundle\), \(\group\)-uniform K-stability is equivalent to \(\group\)-stc K-polystability, which is equivalent to the barycenter condition. 
\end{prop}

\begin{proof}
The result follows from the simple remark that all the combinatorial
data associated to the polarization vary continuously.  One can consequently 
choose a continuous family of elements of the varying moment
polytopes such that it coincides with \(\weightSection\) on the given
\(\lineBundle\).  Then the corresponding function \(L_{r+1}\) varies
continuously as well, and the condition \(\min L_{r+1} > \delta\) is
an open condition.
\end{proof}

\begin{cor}
Let \(\polVar\) be a Gorenstein Fano toroidal horospherical variety. 
Then there exists a neighborhood of the anticanonical line bundle where \(\group\)-uniform K-stability is equivalent to vanishing of the Futaki invariant. 
\end{cor}

\begin{proof}
For any polarized toroidal horospherical variety, the moment polytope does not touch the walls of the positive Weyl chamber of \(\group\) defined by roots not in \(\relevantRoots\). 
Hence the Duistermaat--Heckman polynomial is positive on each moment polytope. 
As a consequence of the proof of Theorem~\ref{thm_Fano}, with the choice \(\weightSection = 2\halfSum_X\), \(L_{r+1}\) is positive on \(\polytope\) for the anticanonical line bundle. 
As a consequence, we can apply Proposition~\ref{prop_open} to obtain that, in a neighborhood of the anticanonical line bundle, \(\group\)-uniform K-stability is equivalent to \(\group\)-stc K-polystability. 
Finally, all special test configurations for horospherical manifolds are product test configurations since the valuation cone is \(\realDualLattice\) in this case. 
Hence \(\group\)-uniform K-stability with respect to special test configurations is equivalent to vanishing of the Futaki invariant. 
\end{proof}

\begin{prop}
Let \(\variety\) be a Gorenstein Fano \(\group\)-spherical polarized variety such that the open orbit \(\homo\) is a non-Hermitian symmetric variety. 
Then on a neighborhood of the anticanonical line bundle, \(\group\)-uniform K-stability is equivalent to \(\group\)-stc K-polystability. 
\end{prop}

\begin{proof}
For a non-Hermitian symmetric space \(\homo\), the valuation cone is the negative Weyl chamber defined by a root system in (a subspace of) \(\realSphericalLattice\) (the restricted root system of the symmetric space), and the images of colors in \(\realDualLattice\) are exactly one positive multiple of each simple coroot of this root system (restricted coroots). 
As a consequence, the outward-pointing normals to facets of  moment polytopes (which are always in \(\realSphericalLattice\) if the symmetric space is \emph{not} Hermitian) are either negative restricted coroots or elements of the positive restricted Weyl chamber. 
Furthermore, the name \emph{restricted} is appropriate in the sense that restricted roots are exactly (doubles of) restrictions of roots in \(\relevantRoots\) to \(\realDualLattice\). 
Finally, the restriction of the Duistermaat--Heckman polynomial to a facet vanishes exactly when the facet is defined by a restricted coroot. 

Consider the anticanonical line bundle on \(\variety\), and instead of \(\weightSection=2\halfSum_X\) as in Section~\ref{sec_Fano}, consider the element \(\weightSection=2t\halfSum_X\). 
It is still in \(\polytope\) for \(t\) close to \(1\) since in our symmetric situation, \(2\halfSum_X\) is the half sum of positive restricted roots and is an element of \(\realSphericalLattice\). 

We then write, still for the anticanonical line bundle, 
\begin{align*} 
L_{r+1}(x) 
& = \left((r+1)n_i^{-1} -2\scalar + \mathrm{Card}\left(\relevantRoots\right)n_i^{-1} \right) \prod_{\alpha\in\relevantRoots} \langle \alpha, x+\weightSection\rangle \\ & \quad + \sum_{\alpha\in\relevantRoots} \left\langle \alpha, \mathrm{Card}\left(\relevantRoots\right)\left( 2 \halfSum_X - n_i^{-1} \weightSection \right)  \right\rangle \prod_{\beta\in \relevantRoots\setminus \{\alpha\}} \langle \beta, x+\weightSection \rangle.
\end{align*}
We know from Section~\ref{sec_Fano} that for every \(i\), 
\[ \left((r+1)n_i^{-1} -2\scalar + \mathrm{Card}\left(\relevantRoots\right)n_i^{-1} \right) \]
is strictly positive if \(t\) is close to \(1\) since these numbers vary continuously with \(t\) and are equal to \(1\) for \(t=1\). 
For the other term, we have
\[ 2 \halfSum_X - n_i^{-1} \weightSection = \left(1-tn_i^{-1}\right) 2\halfSum_X. \]
The values of \(n_i\) depend on \(t\), but
\begin{itemize}
\item if the Duistermaat--Heckman polynomial vanishes on the facet \(E_i\) (\textit{i.e.}, it is defined by a restricted coroot), then we can \emph{choose} the value of \(n_i\) to ensure that \(1-tn_i^{-1}\) is positive for any \(t\); 
\item else, the outward-pointing normal \(u_i\) to the facet \(E_i\) is in the positive restricted Weyl chamber, and \(n_i=1-u_i((t-1)2\halfSum_X)\), so that 
\[ 1-tn_i^{-1} = \frac{(1-t)(1+u_i(2\halfSum_X))}{1-(t-1)u_i(2\halfSum_X)} \]
is positive when \(t<1\). 
\end{itemize}

We can then fix a choice of \(t\) and \(\weightSection=2t\halfSum_X\) so that the corresponding \(L_{r+1}\) is positive on \(\polytope\). 
Applying the same arguments as for Proposition~\ref{prop_open} yields the conclusion.
\end{proof}

\begin{rem}
Let us stress again that it is very likely that the statements proved above hold more generally for spherical varieties. 
It would for example be rather straightforward to push further the last proposition so that it applies to all \(\bbQ\)-Gorenstein weak Fano spherical varieties whose open orbit is affine. 
We leave to further research the exploration of different special cases or the question of finding an argument applying to general spherical varieties. 
\end{rem}

\renewcommand\thesection{\Alph{section}}
\setcounter{section}{0}

\section*{Appendix. Uniform Yau--Tian--Donaldson conjecture for polarized spherical manifolds, by Yuji Odaka\footnote{Department of Mathematics, Kyoto University, Kyoto 606-8285, Japan\\
\hspace*{18pt} \textit{email:} yodaka@math.kyoto-u.ac.jp}}

\addcontentsline{toc}{section}{Appendix. Uniform Yau--Tian--Donaldson conjecture for polarized spherical manifolds, by  Yuji Odaka}
\refstepcounter{section}

The purpose of this short note is to clarify the following statements. 

\begin{thm}\label{fg}
For any $\bbC^{*}$-equivariant isotrivial projective family 
$\pi\colon \mathcal{X}\to \bbP^{1}$ 
whose general fiber is a $G$-spherical projective variety for a reductive algebraic group $G$, take an arbitrary 
line bundle $\mathcal{L}$ which is ample over the general fiber. 
Note that we do not assume it is also ample on the central fiber. 
$($\cite{Li} called $(\mathcal{X},\mathcal{L})$ a {\rm model}.$)$ 

Then, $\oplus_{m\ge 0}\pi_{*} \mathcal{L}^{\otimes m}$ is a finitely generated 
$\mathcal{O}_{\mathbb{P}^{1}}$-algebra. 
\end{thm}

\begin{cor}\label{sph.YTD}
For a polarized smooth projective $G$-spherical varieties $(X,L)$, 
the $G$-uniform K-(poly)stability in the sense of\, \cite{His1, His2, Li} 
implies the existence of a unique cscK metric. 
\end{cor}

\begin{proof}[Proof of Theorem~\ref{fg}]
By applying the Eakin--Nagata theorem to the normalization of $\mathcal{X}$, 
we can and do assume $\mathcal{X}$ is normal. 
Because of the $G$-action and the given compatible $\bbC^*$-action in the horizontal direction 
on $\mathcal{X}$, it follows that $\mathcal{X}$ has the natural structure of  
a $(G\times \bbC^*)$-spherical variety. Indeed, 
the Borel subgroup of $G\times \bbC^{*}$ is simply $B(G)\times \bbC^*$ from the definition, where $B(G)$ denotes the original 
Borel subgroup of $G$, and it admits an open dense orbit inside $X\times \bbC^{*}
(\subset \mathcal{X})$ 
by the $G$-sphericality of $X$. 

By taking a $\bbC^{*}$-equivariant resolution of indeterminacy 
of $\mathcal{X}\dashrightarrow X\times \mathbb{P}^{1}$ 
as in \cite{RT, Od}, we can and do assume $\mathcal{X}$ is the blowup of 
a flag ideal, \textit{i.e.}, of dominating type in the terminology of \cite{Li}. 
Since replacing $\mathcal{L}$ by $\mathcal{L}\otimes p_2^* \mathcal{O}_{\bbP^{1}}(-c)$ 
for $c\in \mathbb{Z}$ does not affect the assertion, we can and do apply such a twist as follows. Here, $p_{i}$ denotes 
\supth{$i$} projection from $X\times \mathbb{P}^{1}$. 

From the finiteness of the irreducible components of $\mathcal{X}_{0}$, 
it easily follows that there exists a large enough $c'\in \mathbb{Z}_{>0}$ such that 
$$p_1^* L \otimes p_2^* \mathcal{O}_{\bbP^{1}}(-c')
\subset \mathcal{L}
\subset p_1^* L \otimes p_2^* \mathcal{O}_{\bbP^{1}}(c').$$ 
Twisting the above by $\mathcal{O}_{\bbP^{1}}(c')$ and letting $c:=2c'$, 
we can and do assume 
\begin{align}\label{inclusion}
p_1^* L \subset \mathcal{L} \subset 
p_1^* L \otimes p_2^* \mathcal{O}_{\bbP^{1}}(c).\end{align}
In any case, it immediately follows from the above that 
the filtration associated to 
$\mathcal{L}$ is linearly bounded in the sense of \cite{Sz15}. 

Since we confirmed that $\mathcal{X}$ is a spherical variety, 
it is also a Mori dream space in the sense of \cite{HK} due to \cite{BK} (\textit{cf.}\ \cite[p.~340]{HK}). 
Therefore, it follows that 
$\oplus_{m\ge 0} H^0(\mathcal{X}, \mathcal{L}^{\otimes m})$ is a finitely generated 
graded $\bbC$-algebra. 
Then, by the natural $\bbC^{*}$-action on it which is induced by 
the $\mathbb{G}_{m}$-action on $(\mathcal{X},\mathcal{L})$, 
the complete reductivity of $\mathbb{G}_{m}$ implies that 
we can take a set of finite generators as eigenvectors 
of the form 
$\mathcal{S}=\{(p_{1}^{*}s_{i}^{(m,l)})t^{l}\}_{l\le 0,m,i}$. 
Here, $t$ denotes the homogeneous coordinate of $\mathbb{P}^{1}$ 
which vanishes at the origin with order one, which we also 
identify with $p_{2}^{*}t$ on $\mathcal{X}$, and 
the indices are of the form  
$-cm\le l\le 0$ and $1\le i\le a_{m,l}$ for a double sequence of 
positive integers $a_{m,l}$ such that 
$$0=a_{m,-cm-1}\le a_{m,-cm}\le \cdots \le 
a_{m,0}=a_{m,1}=\cdots=h^{0}(X,L^{\otimes m})$$
because of \eqref{inclusion}. 
(In the case of test configurations, \textit{i.e.}, when $\mathcal{L}$ is 
relatively ample, 
these $\{a_{m,l}\}$ are determined by $\{\lambda_{a,k}\}$ 
in the terminology of Section~\ref{sec_stability} of this paper.) 
The proof follows from standard arguments, so we omit it, but 
see and compare with \cite{Od} for instance. 
Similarly, the 
set
\begin{align*}
(\mathcal{S}\subset\,)\tilde{\mathcal{S}}&=
\left\{\left(p_{1}^{*}s_{i}^{(m,\min\{l,0\})}\right)t^{l}\right\}_{m,l\in \mathbb{Z},i}, 
\end{align*}
in which we allow the integer index 
$l$ to be not necessarily negative, generates  
$\oplus_{m\ge 0}\pi_{*}\mathcal{L}^{\otimes m}$ as a 
$\mathcal{O}_{\mathbb{P}^{1}}$-algebra.
Note that 
\begin{align*}
\tilde{\mathcal{S}}
=\mathcal{S}\sqcup \left\{\left\{(p_{1}^{*}s_{i}^{(m,0)})t^{l}\right\}_{l> 0,m,i}\right\}
=\mathcal{S}\sqcup \left(t\bbC[t] \left\{\left\{(p_{1}^{*}s_{i}^{(m,0)})\right\}_{m,i}\right\}\right). 
\end{align*}
Thus, 
we can particularly take its finite subset 
$$\overline{\mathcal{S}}:=\mathcal{S}\sqcup 
\left\{\left\{(p_{1}^{*}s_{i}^{(m,0)})\right\}_{m,i} (\subset \tilde{\mathcal{S}})\right\},$$ 
which still generates 
$\oplus_{m\ge 0}\pi_{*}\mathcal{L}^{\otimes m}$ as a graded 
$\mathcal{O}_{\mathbb{P}^{1}}$-algebra. 
This completes the proof of Theorem~\ref{fg}. 
\end{proof}

\begin{proof}[Proof of Corollary~\ref{sph.YTD}]
  The result of \cite[Theorem~1.10]{Li}
  combined with  Theorem~\ref{fg} 
readily imply the existence part of Corollary~\ref{sph.YTD}. 
The uniqueness part is due to \cite{BB} for general cscK metrics. 
\end{proof}

Note that in the toric case, \textit{i.e.}, when $G$ is an algebraic
torus, Corollary~\ref{sph.YTD} was known before as a result of
\cite{His1} combined with \cite{CC, CC2}.  Our approach above extends
\cite[Theorem~1.12]{Li} by some part of the theory of the Mori dream
space \cite{BK, HK}.


\newcommand{\etalchar}[1]{$^{#1}$}
\providecommand{\bysame}{\leavevmode\hbox to3em{\hrulefill}\thinspace}
\providecommand{\MR}{\relax\ifhmode\unskip\space\fi MR }
\providecommand{\MRhref}[2]{%
  \href{http://www.ams.org/mathscinet-getitem?mr=#1}{#2}
}
\providecommand{\href}[2]{#2}

\end{document}